\documentclass[a4paper,
fontsize=11pt,%
oneside,%
numbers=enddot]{scrartcl}
\KOMAoptions{DIV=12}    
\usepackage[utf8]{inputenc}
\usepackage[T1]{fontenc}
\usepackage{textcomp}
\usepackage[fleqn]{amsmath}
\usepackage{amsthm}
\usepackage{amssymb}
\usepackage{amsfonts}
\usepackage{libertine}
\usepackage[libertine,
slantedGreek,
nosymbolsc,
nonewtxmathopt,
subscriptcorrection]{newtxmath}
\usepackage[scaled=0.95,varqu,varl]{inconsolata}
\usepackage[scr=boondox]{mathalpha}   
\usepackage{euscript}   
\usepackage{leftindex}
\allowdisplaybreaks[1]  
\numberwithin{equation}{section}    
\usepackage[svgnames,hyperref]{xcolor}
\definecolor{dblue}{HTML}{0455BF}
\definecolor{dgreen}{HTML}{02724A}
\definecolor{dgreen2}{HTML}{025951}
\definecolor{dred}{HTML}{D90404}
\definecolor{dviolet}{HTML}{42208C}
\definecolor{labelkey}{HTML}{025951}
\definecolor{refkey}{HTML}{025951}
\addtokomafont{section}{\centering}

\usepackage{enumitem}
\setlist{itemsep=-2.0pt}

\usepackage{upref}  
\usepackage{hyperref}
\hypersetup{colorlinks=true,
linktocpage=true,
linkcolor=dblue,
citecolor=dgreen,
urlcolor=dred,
pdfencoding=auto,
hypertexnames=false}
\makeatletter
\g@addto@macro\th@plain{
\thm@headfont{\bfseries\sffamily}
\thm@notefont{}}
\g@addto@macro\th@definition{
\thm@headfont{\bfseries\sffamily}
\thm@notefont{}}
\g@addto@macro\th@remark{
\thm@headfont{\bfseries\sffamily}
\thm@notefont{}}
\makeatother
\theoremstyle{plain}
\newtheorem{theorem}{Theorem}[section]
\newtheorem{proposition}[theorem]{Proposition}
\newtheorem{corollary}[theorem]{Corollary}
\newtheorem{lemma}[theorem]{Lemma}

\theoremstyle{definition}
\newtheorem{definition}[theorem]{Definition}
\newtheorem{example}[theorem]{Example}
\newtheorem{problem}[theorem]{Problem}
\newtheorem{assumption}[theorem]{Assumption}
\theoremstyle{remark}
\newtheorem{remark}[theorem]{Remark}
\newtheorem{notation}[theorem]{Notation}

\usepackage[extdef=true]{delimset}
\DeclareMathDelimiterSet{\scal}[2]{
\selectdelim[l]<{#1}
\mathpunct{}\selectdelim[p]|
{#2}\selectdelim[r]>}

\newcommand{\menge}[2]{\bigl\{{#1}\mid{#2}\bigr\}} 
\DeclareMathDelimiterSet{\Menge}[2]{\selectdelim[l]\{
{#1}\selectdelim[m]|{#2}\selectdelim[r]\}}
\newcommand*\Cdot{{\mkern 1.6mu\cdot\mkern 1.6mu}}
\makeatletter
\def\upintkern@{\mkern-7mu\mathchoice{\mkern-3.5mu}{}{}{}}
\def\upintdots@{\mathchoice{\mkern-4mu\@cdots\mkern-4mu}%
{{\cdotp}\mkern1.5mu{\cdotp}\mkern1.5mu{\cdotp}}%
{{\cdotp}\mkern1mu{\cdotp}\mkern1mu{\cdotp}}%
{{\cdotp}\mkern1mu{\cdotp}\mkern1mu{\cdotp}}}
\makeatother
\DeclareFontFamily{OMX}{mdbch}{}
\DeclareFontShape{OMX}{mdbch}{m}{n}{ <->s * [0.8]  mdbchr7v }{}
\DeclareFontShape{OMX}{mdbch}{b}{n}{ <->s * [0.8]  mdbchb7v }{}
\DeclareFontShape{OMX}{mdbch}{bx}{n}{<->ssub * mdbch/b/n}{}
\DeclareSymbolFont{uplargesymbols}{OMX}{mdbch}{m}{n}
\SetSymbolFont{uplargesymbols}{bold}{OMX}{mdbch}{b}{n}
\DeclareMathSymbol{\upintop}{\mathop}{uplargesymbols}{82}
\DeclareMathSymbol{\upointop}{\mathop}{uplargesymbols}{"48}
\makeatletter
\renewcommand{\int}{\DOTSI\upintop\ilimits@}
\renewcommand{\oint}{\DOTSI\upointop\ilimits@}
\makeatother

\newcommand{\qq}{\mathscr{Q}}
\newcommand{\RR}{\mathbb{R}}
\newcommand{\NN}{\mathbb{N}}
\newcommand{\HH}{\mathcal{H}}

\newcommand{\GG}{\mathfrak{H}}
\newcommand{\AW}{\mathsf{A}_{\omega}}

\newcommand{\HW}{\mathsf{H}_{\omega}}
\newcommand{\TW}{\mathsf{T}_{\omega}}
\newcommand{\RW}{\mathsf{R}_{\omega}}
\newcommand{\LW}{\mathsf{L}_{\omega}}
\newcommand{\fw}{\mathsf{f}_{\omega}}
\newcommand{\AS}{\mathsf{A}}
\newcommand{\HS}{\mathsf{H}}
\newcommand{\LS}{\mathsf{L}}
\newcommand{\XS}{\mathsf{X}}
\newcommand{\TS}{\mathsf{T}}
\newcommand{\BE}{\EuScript{B}}
\newcommand{\FF}{\EuScript{F}}

\newcommand{\pinf}{{+}\infty}
\newcommand{\minf}{{-}\infty}

\newcommand{\RXX}{\intv{\minf}{\pinf}}
\newcommand{\RX}{\intv[l]0{\minf}{\pinf}}

\newcommand{\RPP}{\intv[o]0{0}{\pinf}}

\newcommand{\emp}{\varnothing}
\newcommand{\forallmu}{\forall^{\mu}}

\newcommand{\Int}{\displaystyle\int}

\newcommand{\infconv}{\mathbin{\mbox{\small$\square$}}}

\newcommand{\pushfwd}%
{\ensuremath{\mbox{\Large$\,\triangleright\,$}}}

\DeclareMathOperator{\Argmin}{Argmin}

\newcommand{\Id}{\mathrm{Id}}
\newcommand{\moyo}[2]{\leftindex[I]^{#2}{#1}}

\DeclareMathOperator{\cran}{\overline{ran}}
\DeclareMathOperator{\ran}{ran}
\DeclareMathOperator{\cdom}{\overline{dom}}
\DeclareMathOperator{\dom}{dom}
\DeclareMathOperator{\intdom}{int\,dom}

\DeclareMathOperator{\esssup}{ess\,sup}

\DeclareMathOperator{\gra}{gra}
\DeclareMathOperator{\zer}{zer}
\DeclareMathOperator{\inte}{int}

\DeclareMathOperator{\prox}{prox}
\DeclareMathOperator{\proj}{proj}
\DeclareMathOperator{\spc}{\overline{span}}
\newcommand{\EE}{\mathsf{E}}
\newcommand{\PP}{\mathsf{P}}
\DeclareMathOperator{\PE}{\overset{\diamond}{\mathsf{E}}}
\newcommand{\proxc}[2]{{#1}\diamond{#2}}
\DeclareFontFamily{U}{mathb}{}
\DeclareFontShape{U}{mathb}{m}{n}{<-5.5> mathb5 <5.5-6.5> mathb6 
<6.5-7.5> mathb7 <7.5-8.5> mathb8 <8.5-9.5> mathb9 <9.5-11> mathb10
<11-> mathb12}{}
\DeclareSymbolFont{mathb}{U}{mathb}{m}{n}
\DeclareFontSubstitution{U}{mathb}{m}{n}
\DeclareMathSymbol{\blackdiamond}{\mathbin}{mathb}{"0C}
\newcommand{\proxcc}[2]{{#1}\mathbin{\blackdiamond}{#2}}
\DeclareMathOperator{\rmi}{\overset{\mathord{\diamond}}{\mathsf{M}}}
\DeclareMathOperator{\rcm}{\overset{\mathord{\blackdiamond}}%
{\mathsf{M}}}

\renewcommand{\leq}{\leqslant}
\renewcommand{\geq}{\geqslant}

\newcommand{\mae}{\text{\normalfont$\mu$-a.e.}}

\renewenvironment{abstract}{%
\vspace*{-0.50cm}
\small
\quotation%
\noindent%
{\normalfont\bfseries\sffamily
\nobreak\abstractname\ }%
}{%
\endquotation%
\medskip
}
\renewcommand{\abstractname}{Abstract.}
\newcommand\keywordsname{Keywords.}
\newenvironment{keywords}
{\renewcommand\abstractname{\keywordsname}\begin{abstract}}
{\end{abstract}}
\usepackage[auth-sc]{authblk}
\newcommand{\email}[1]{\href{mailto:#1}{\nolinkurl{#1}}}
\renewcommand*\Affilfont{\normalfont\normalsize}
\newcommand\affilcr{\protect\\ \protect\Affilfont}
\makeatletter
\renewcommand\AB@affilsepx{\protect\\[0.5em]}
\makeatother
\author[1]{Minh N. B\`ui}
\affil[1]{University of Graz
\affilcr
Department of Mathematics and Scientific Computing, NAWI Graz
\affilcr
8010 Graz, Austria
\affilcr
\email{minh.bui@uni-graz.at}
}
\author[2]{Patrick L. Combettes}
\affil[2]{North Carolina State University
\affilcr
Department of Mathematics
\affilcr
Raleigh, NC 27695, USA
\affilcr
\email{plc@math.ncsu.edu}
}

\begin{document}

\title{%
Integral Resolvent and Proximal Mixtures\thanks{%
Contact author: P. L. Combettes.
Email: \email{plc@math.ncsu.edu}.
Phone: +1 919 515 2671.
The work of P. L. Combettes was supported by the National
Science Foundation under grant CCF-2211123. 
}}

\date{~}

\maketitle

\centerline{\textit{Dedicated to the memory of Boris T. Polyak}}

\vspace{12mm}

\begin{abstract}
Using the theory of Hilbert direct integrals, we introduce and
study a monotonicity-preserving operation, termed the integral
resolvent mixture. It combines arbitrary families of monotone
operators acting on different spaces and linear operators. As a
special case, we investigate the resolvent expectation, an
operation which combines monotone operators in such a way that the
resulting resolvent is the Lebesgue expectation of the individual
resolvents. Along the same lines, we introduce an operation that
mixes arbitrary families of convex functions defined on different
spaces and linear operators to create a composite convex function.
Such constructs have so far been limited to finite families of
operators and functions. The subdifferential of the integral
proximal mixture is shown to be the integral resolvent mixture of
the individual subdifferentials. Applications to the relaxation of
systems of composite monotone inclusions are presented.
\end{abstract}

\begin{keywords}
Hilbert direct integral,
integral proximal mixture,
integral resolvent mixture,
monotone operator,
proximal expectation,
resolvent expectation.
\end{keywords}

\vspace*{-0.50cm}
\small
\quotation%
\noindent%
{\normalfont%
Communicated by Stefan Ulbrich.
\nobreak}

\newpage

\section{Introduction}

Monotone inclusions provide an effective template to model a wide
spectrum of problems in optimization and nonlinear analysis
\cite{Livre1,Brez98,Bri18b,Bord18,Acnu24,Facc03,Rock84}.
The question of
combining monotone and linear operators in a fashion that preserves
monotonicity has been a recurrent topic; see, e.g.,
\cite{Beck14,Botr10,Bric23,Brow69,Svva23}. One such construct is
the resolvent mixture \cite{Svva23}, an operation that includes in
particular the resolvent average \cite{Baus16}. It combines
finitely many monotone and linear operators in such a way that the
resolvent of the resulting operator is the sum of the individual
linearly composed resolvents. Our objective is to extend this
construct to arbitrary families of operators. Our analysis rests on
the concept of Hilbert direct integrals of families of monotone
operators proposed in \cite{Cana24}. Considering the case when the
underlying operators are subdifferentials leads us to introduce
the Hilbert direct integral of a family of convex functions, a
notion that extends proximal mixtures of finite families and, 
in particular, the proximal average.

Our main contributions are the following.
\begin{itemize}
\item
We introduce the notion of an integral resolvent mixture for
arbitrary families of monotone operators acting on different
spaces. This construction exploits the notion of
Hilbert direct integrals of set-valued and linear operators from
\cite{Cana24}. One of its salient features is that its resolvent is
the Lebesgue integral of the linearly composed resolvents of the
individual operators. A dual operation of integral resolvent
comixture is also investigated.
\item
We introduce the notion of an integral proximal mixture for
arbitrary families of functions defined on different spaces. Its
proximity operator turns out to be the Lebesgue integral of the
linearly composed proximity operators of the individual functions.
A dual operation of integral proximal comixture is also
investigated.
\item
As an instance of an integral resolvent mixture, we propose a
notion of resolvent expectation for a family of maximally monotone
operators and, likewise, of proximal expectation for a family of
functions. These notions extend those of resolvent and
proximal averages for finite families. 
\item
We apply the above tools to the relaxation of systems of 
monotone inclusions involving linear operators. Applications
fitting this framework are described and a proximal-type algorithm
is proposed.  
\end{itemize}

The paper is organized as follows. In Section~\ref{sec:2}, we set
our notation and provide necessary theoretical tools. In 
Section~\ref{sec:3}, we study the integral resolvent mixture 
of a family of monotone operators. Section~\ref{sec:4}
is dedicated to the integral proximal mixture of a family of 
functions. In Section~\ref{sec:5}, we present an application to 
systems of monotone inclusions and discuss some special cases of
interest arising in data analysis.

\section{Notation and background}
\label{sec:2}

We first present our notation, which follows \cite{Livre1}.

Let $\HH$ be a real Hilbert space with power set $2^{\HH}$, 
identity operator $\Id_{\HH}$, scalar product 
$\scal{\Cdot}{\Cdot}_{\HH}$, associated norm 
$\norm{\Cdot}_{\HH}$, and quadratic kernel
$\qq_{\HH}=\norm{\Cdot}_{\HH}^2/2$.

Let $C$ be a nonempty closed convex subset of $\HH$. Then $\proj_C$
is the projection operator onto $C$ and $N_C$ is the normal cone 
operator of $C$.

Let $T\colon\HH\to\HH$ and $\tau\in\RPP$. Then $T$ is nonexpansive
if it is $1$-Lipschitzian, $\tau$-cocoercive if 
\begin{equation}
\label{e:coco}
(\forall x\in\HH)(\forall y\in\HH)\quad
\scal{x-y}{Tx-Ty}_{\HH}\geq\tau\norm{Tx-Ty}_{\HH}^2,
\end{equation}
and $T$ is firmly nonexpansive if it is $1$-cocoercive.

Let $A\colon\HH\to 2^{\HH}$. The graph of $A$ is 
$\gra A=\menge{(x,x^*)\in\HH\times\HH}{x^*\in Ax}$, 
the inverse of $A$ is the operator 
$A^{-1}\colon\HH\to 2^{\HH}$ with graph
$\gra A^{-1}=\menge{(x^*,x)\in\HH\times\HH}{x^*\in Ax}$,
the domain of $A$ is $\dom A=\menge{x\in\HH}{Ax\neq\emp}$, 
the range of $A$ is $\ran A=\bigcup_{x\in\dom A}Ax$, 
the set of zeros of $A$ is
$\zer A=\menge{x\in\HH}{0\in Ax}$, 
the resolvent of $A$ is $J_A=(\Id_{\HH}+A)^{-1}$,
and the Yosida approximation of $A$ of index $\gamma\in\RPP$ is
\begin{equation}
\moyo{A}{\gamma}=A\infconv\brk1{\gamma^{-1}\Id_{\HH}}
=\brk1{A^{-1}+\gamma\Id_{\HH}}^{-1}
=\dfrac{\Id_{\HH}-J_{\gamma A}}{\gamma}.
\end{equation}
Suppose that $A$ is monotone. Then $A$ is maximally monotone if any
extension of $\gra A$ is no longer monotone in $\HH\oplus\HH$. In
this case, $\dom J_A=\HH$ and $J_A$ is firmly nonexpansive.

Let $f\colon\HH\to\RXX$ and set 
$\dom f=\menge{x\in\HH}{f(x)<\pinf}$. The Moreau envelope of $f$ is 
\begin{equation}
f\infconv\qq_{\HH}\colon\HH\to\RXX\colon x\mapsto
\inf_{y\in\HH}\brk1{f(y)+\qq_{\HH}(x-y)}
\end{equation}
and the conjugate of $f$ is 
\begin{equation}
f^*\colon\HH\to\RXX\colon x^*\mapsto\sup_{x\in\HH}
\bigl(\scal{x}{x^*}_{\HH}-f(x)\bigr).
\end{equation}
Now suppose that $f\in\Gamma_0(\HH)$, that is, $f$ is lower
semicontinuous, convex, and such that 
$\minf\notin f(\HH)\neq\{\pinf\}$. The subdifferential of $f$ is
the maximally monotone operator
\begin{equation}
\label{e:subdiff}
\partial f\colon\HH\to 2^{\HH}\colon x\mapsto\menge{x^*\in\HH}
{(\forall y\in\HH)\,\,\scal{y-x}{x^*}_{\HH}+f(x)\leq f(y)}
\end{equation}
and the proximity operator $\prox_f=J_{\partial f}$ of $f$ maps
every $x\in\HH$ to the unique minimizer of the function
$\HH\to\RX\colon y\mapsto f(y)+\qq_{\HH}(x-y)$.

Finally, given a measure space $(\Omega,\FF,\mu)$,
the symbol $\forallmu$ means ``for $\mu$-almost every''
\cite{Sch93b}.

\begin{definition}[\protect{\cite[Definition~1.1]{Svva23}}]
\label{d:s1}
Let $\HH$ and $\XS$ be real Hilbert spaces,
let $A\colon\HH\to 2^{\HH}$,
and let $L\colon\XS\to\HH$ be linear and bounded.
The \emph{resolvent composition} of $A$ with $L$ is the operator
$\proxc{L}{A}\colon\XS\to 2^{\XS}$ given by 
\begin{equation}
\label{e:s1}
\proxc{L}{A}=\brk!{L^*\circ J_A\circ L}^{-1}-\Id_{\XS}
\end{equation}
and the \emph{resolvent cocomposition} of $A$ with $L$ is 
$\proxcc{L}{A}=(\proxc{L}{A^{-1}})^{-1}$.
\end{definition}

\begin{definition}[\protect{\cite[Definition~1.4]{Svva23}}]
\label{d:s2}
Let $\HH$ and $\XS$ be real Hilbert spaces, let 
$f\colon\HH\to\RXX$, and let $L\colon\XS\to\HH$ be linear and
bounded. The \emph{proximal composition} of $f$ with $L$ is the
function $\proxc{L}{f}\colon\XS\to\RXX$ given by 
\begin{equation}
\label{e:s2}
\proxc{L}{f}=\brk!{(f^*\infconv\qq_{\mathcal{H}})\circ L}^*-
\qq_{\XS},
\end{equation}
and the \emph{proximal cocomposition} of $f$ with $L$ is 
$\proxcc{L}{f}=(\proxc{L}{f^*})^*$.
\end{definition}

Here are some notation and facts regarding integration in Hilbert
spaces. Let $(\Omega,\FF,\mu)$ be a complete $\sigma$-finite
measure space and let $\HS$ be a separable real Hilbert
space. For every $p\in\intv[r]{1}{\pinf}$, set
\begin{equation}
\mathscr{L}^p\brk!{\Omega,\FF,\mu;\HS}
=\Menge3{x\colon\Omega\to\HS}{x\,\,
\text{is $(\FF,\BE_{\HS})$-measurable and}\,\,
\int_{\Omega}\norm{x(\omega)}_{\HS}^p\,\mu(d\omega)<\pinf},
\end{equation}
where $\BE_{\HS}$ is the Borel $\sigma$-algebra of $\HS$.
The Lebesgue (also called Bochner) integral of a
mapping $x\in\mathscr{L}^1(\Omega,\FF,\mu;\HS)$ is denoted by
$\int_{\Omega}x(\omega)\mu(d\omega)$; see
\cite[Section~V.{\S}7]{Sch93b} for background. 
We denote by $L^p\brk{\Omega,\FF,\mu;\HS}$ the space of 
equivalence classes of $\mae$ equal mappings in 
$\mathscr{L}^p\brk{\Omega,\FF,\mu;\HS}$.

\begin{lemma}
\label{l:5}
Let $(\Omega,\FF,\mu)$ be a complete $\sigma$-finite measure space,
let $\HS$ be a separable real Hilbert space,
and let $x\in\mathscr{L}^1(\Omega,\FF,\mu;\HS)$.
Then the following hold:
\begin{enumerate}
\item
\label{l:5i}
\textup{\cite[Th\'eor\`eme~5.7.13]{Sch93b}}
$\norm{\int_{\Omega}x(\omega)\mu(d\omega)}_{\HS}\leq
\int_{\Omega}\norm{x(\omega)}_{\HS}\,\mu(d\omega)$.
\item
\label{l:5ii}
Let $\mathsf{x}^*\in\HS$. Then the function
$\Omega\to\RR\colon\omega\mapsto
\scal{x(\omega)}{\mathsf{x}^*}_{\HS}$ is $\mu$-integrable and
\begin{equation}
\int_{\Omega}\scal{x(\omega)}{\mathsf{x}^*}_{\HS}\,\mu(d\omega)
=\scal3{\int_{\Omega}x(\omega)\mu(d\omega)}{\mathsf{x}^*}_{\HS}.
\end{equation}
\item
\label{l:5iii}
Suppose that $\mu$ is a probability measure. Then
\begin{equation}
\norm3{\int_{\Omega}x(\omega)\mu(d\omega)}_{\HS}^2\leq
\int_{\Omega}\norm{x(\omega)}_{\HS}^2\,\mu(d\omega).
\end{equation}
\end{enumerate}
\end{lemma}
\begin{proof}
\ref{l:5ii}:
Apply \cite[Th\'eor\`eme~5.8.16]{Sch93b} with the continuous linear
functional $L=\scal{\Cdot}{\mathsf{x}^*}_{\HS}$.

\ref{l:5iii}:
We derive from \ref{l:5i} and the Cauchy--Schwarz inequality that
\begin{equation}
\norm3{\int_{\Omega}x(\omega)\mu(d\omega)}_{\HS}^2
\leq\abs3{\int_{\Omega}1_{\Omega}(\omega)
\norm{x(\omega)}_{\HS}\,\mu(d\omega)}^2
\leq\mu(\Omega)
\int_{\Omega}\norm{x(\omega)}_{\HS}^2\,\mu(d\omega),
\end{equation}
which concludes the proof.
\end{proof}

\begin{notation}
\label{n:62}
Let $(\Omega,\FF,\mu)$ be a complete $\sigma$-finite measure
space, let $\XS$ and $\HS$ be separable real Hilbert spaces,
and let $(\TW)_{\omega\in\Omega}$ be a family of operators from
$\XS$ to $\HS$ such that, for every $\mathsf{x}\in\XS$,
the mapping $\Omega\to\HS\colon\omega\mapsto\TW\mathsf{x}$
is $(\FF,\BE_{\HS})$-measurable. Let
\begin{equation}
\mathsf{D}=\Menge3{\mathsf{x}\in\XS}{
\int_{\Omega}\norm{\TW\mathsf{x}}_{\HS}\,\mu(d\omega)<\pinf}.
\end{equation}
Then
\begin{equation}
\int_{\Omega}\TW\mu(d\omega)\colon\mathsf{D}\to\HS\colon
\mathsf{x}\mapsto\int_{\Omega}\TW\mathsf{x}\,\mu(d\omega).
\end{equation}
In particular, if $\mu$ is a probability measure, then
\begin{equation}
\label{e:vk}
\EE(\TW)_{\omega\in\Omega}=\int_{\Omega}\TW\mu(d\omega)
\end{equation}
is the \emph{$\mu$-expectation} of the family
$(\TW)_{\omega\in\Omega}$.
\end{notation}

The following setup describes the main functional setting employed
in the paper. As in \cite{Cana24}, it relies on the notion of a
Hilbert direct integral of Hilbert spaces \cite{Dixm69}.

\begin{assumption}
\label{a:1}
Let $(\Omega,\FF,\mu)$ be a complete $\sigma$-finite measure space,
let $(\HW)_{\omega\in\Omega}$ be a family of real Hilbert spaces,
and let $\prod_{\omega\in\Omega}\HW$ be the usual real vector
space of mappings $x$ defined on $\Omega$ such that
$(\forall\omega\in\Omega)$ $x(\omega)\in\HW$.
Let $((\HW)_{\omega\in\Omega},\mathfrak{G})$ be an
\emph{$\FF$-measurable vector field of real Hilbert spaces},
that is, $\mathfrak{G}$ is a vector subspace of
$\prod_{\omega\in\Omega}\HW$ which satisfies the following:
\begin{enumerate}[label={\normalfont[\Alph*]}]
\item
\label{a:1A}
For every $x\in\mathfrak{G}$, the function $\Omega\to\RR\colon
\omega\mapsto\norm{x(\omega)}_{\HW}$ is
$\FF$-measurable.
\item
\label{a:1B}
For every $x\in\prod_{\omega\in\Omega}\HW$,
\begin{equation}
\brk[s]!{\;(\forall y\in\mathfrak{G})\;\;
\Omega\to\RR\colon
\omega\mapsto\scal{x(\omega)}{y(\omega)}_{\HW}
\,\,\text{is $\FF$-measurable}\;}
\quad\Rightarrow\quad x\in\mathfrak{G}.
\end{equation}
\item
\label{a:1C}
There exists a sequence $(e_n)_{n\in\NN}$ in $\mathfrak{G}$ such
that $(\forall\omega\in\Omega)$
$\spc\{e_n(\omega)\}_{n\in\NN}=\HW$.
\end{enumerate}
Set
\begin{equation}
\label{e:4a5t}
\GG=\Menge3{x\in\mathfrak{G}}{\int_{\Omega}
\norm{x(\omega)}_{\HW}^2\mu(d\omega)<\pinf},
\end{equation}
and let $\HH$ be the real Hilbert space of equivalence classes of
$\mae$ equal mappings in $\GG$ equipped with the scalar product
\begin{equation}
\label{e:y9d2}
\scal{\Cdot}{\Cdot}_{\HH}\colon\HH\times\HH\to\RR\colon
(x,y)\mapsto\int_{\Omega}
\scal{x(\omega)}{y(\omega)}_{\HW}\mu(d\omega),
\end{equation}
where we adopt the common practice of designating by $x$ both an
equivalence class in $\HH$ and a representative of it in $\GG$.
We write
\begin{equation}
\HH=\leftindex^{\mathfrak{G}}{\int}_{\Omega}^{\oplus}
\HW\mu(d\omega)
\end{equation}
and call $\HH$ the \emph{Hilbert direct integral of
$((\HW)_{\omega\in\Omega},\mathfrak{G})$} \cite{Dixm69}.
\end{assumption}

Here are some instances of Hilbert direct integrals \cite{Cana24}.

\begin{example}
\label{ex:1i+}
Let $p\in\NN\smallsetminus\{0\}$,
let $(\alpha_k)_{1\leq k\leq p}$ be a family in $\RPP$,
let $(\HS_k)_{1\leq k\leq p}$ be separable real Hilbert spaces,
let $\mathfrak{G}=\HS_1\times\cdots\times\HS_p$ be
the usual Cartesian product vector space, and set
\begin{equation}
\Omega=\set{1,\ldots,p},\quad
\FF=2^{\set{1,\ldots,p}},\quad\text{and}\quad
\brk!{\forall k\in\set{1,\ldots,p}}\;\;
\mu\brk!{\set{k}}=\alpha_k.
\end{equation}
Then $((\HS_k)_{1\leq k\leq p},\mathfrak{G})$ is an
$\FF$-measurable vector field of real Hilbert spaces
and $\leftindex^{\mathfrak{G}}{\int}_{\Omega}^{\oplus}
\HW\mu(d\omega)$ is the weighted Hilbert direct sum
of $(\HS_k)_{1\leq k\leq p}$, namely the Hilbert space
obtained by equipping $\mathfrak{G}$ with the scalar product
\begin{equation}
\brk!{(\mathsf{x}_k)_{1\leq k\leq p},
(\mathsf{y}_k)_{1\leq k\leq p}}\mapsto
\sum_{k=1}^p\alpha_k
\scal{\mathsf{x}_k}{\mathsf{y}_k}_{\HS_k}.
\end{equation}
\end{example}

\begin{example}
\label{ex:1ii}
Let $(\alpha_k)_{k\in\NN}$ be a family in $\RPP$,
let $(\HS_k)_{k\in\NN}$ be separable real Hilbert spaces,
let $\mathfrak{G}=\prod_{k\in\NN}\HS_k$, and set
\begin{equation}
\Omega=\NN,\quad\FF=2^{\NN},\quad\text{and}\quad
(\forall k\in\NN)\;\;\mu\brk!{\set{k}}=\alpha_k.
\end{equation}
Then $((\HS_k)_{k\in\NN},\mathfrak{G})$ is an
$\FF$-measurable vector field of real Hilbert spaces and
$\leftindex^{\mathfrak{G}}{\int}_{\Omega}^{\oplus}
\HW\mu(d\omega)$ is the Hilbert space obtained
by equipping the vector space
\begin{equation}
\mathfrak{H}=\Menge3{(\mathsf{x}_k)_{k\in\NN}
\in\mathfrak{G}}{\sum_{k\in\NN}\alpha_k
\norm{\mathsf{x}_k}_{\HS_k}^2<\pinf}
\end{equation}
with the scalar product
\begin{equation}
\brk!{(\mathsf{x}_k)_{k\in\NN},(\mathsf{y}_k)_{k\in\NN}}
\mapsto\sum_{k\in\NN}\alpha_k
\scal{\mathsf{x}_k}{\mathsf{y}_k}_{\HS_k}.
\end{equation}
\end{example}

\begin{example}
\label{ex:1iii}
Let $(\Omega,\FF,\mu)$ be a complete $\sigma$-finite measure space,
let $\HS$ be a separable real Hilbert space, and set
\begin{equation}
\brk[s]!{\;(\forall\omega\in\Omega)\;\;\HW=\HS\;}
\quad\text{and}\quad
\mathfrak{G}=\menge{x\colon\Omega\to\HS}{
x\,\,\text{is $(\FF,\BE_{\HS})$-measurable}}.
\end{equation}
Then $((\HW)_{\omega\in\Omega},\mathfrak{G})$ is an
$\FF$-measurable vector field of real Hilbert spaces and
\begin{equation}
\leftindex^{\mathfrak{G}}{\int}_{\Omega}^{\oplus}
\HW\mu(d\omega)
=L^2\brk!{\Omega,\FF,\mu;\HS}.
\end{equation}
\end{example}

\section{Integral resolvent mixtures}
\label{sec:3}

Our setting hinges on the following assumptions.

\begin{assumption}
\label{a:2}
Assumption~\ref{a:1} and the following are in force:
\begin{enumerate}[label={\normalfont[\Alph*]}]
\item
\label{a:2a}
For every $\omega\in\Omega$, $\AW\colon\HW\to 2^{\HW}$
is maximally monotone.
\item
\label{a:2b}
For every $x\in\GG$, the mapping $\omega\mapsto
J_{\AW}x(\omega)$ lies in $\mathfrak{G}$.
\item
\label{a:2c}
$\dom\leftindex^{\mathfrak{G}}{\int}_{\Omega}^{\oplus}
\AW\mu(d\omega)\neq\emp$, where
\begin{equation}
\label{e:8t0r}
\leftindex^{\mathfrak{G}}{\int}_{\Omega}^{\oplus}
\AW\mu(d\omega)\colon\HH\to 2^{\HH}\colon x\mapsto
\menge{x^*\in\HH}{(\forallmu\omega\in\Omega)\,\,
x^*(\omega)\in\AW x(\omega)}
\end{equation}
is the \emph{Hilbert direct integral of the operators
$(\AW)_{\omega\in\Omega}$ relative to $\mathfrak{G}$}
\cite{Cana24}.
\end{enumerate}
\end{assumption}

\begin{assumption}
\label{a:4}
Assumption~\ref{a:1} and the following are in force:
\begin{enumerate}[label={\normalfont[\Alph*]}]
\item
\label{a:4a}
$\XS$ is a separable real Hilbert space.
\item
\label{a:4b}
For every $\omega\in\Omega$, $\LW\colon\XS\to\HW$ is
linear and bounded.
\item
\label{a:4c}
For every $\mathsf{x}\in\XS$, the mapping
$\mathfrak{e}_{\LS}\mathsf{x}\colon
\omega\mapsto\LW\mathsf{x}$ lies in $\mathfrak{G}$.
\item
\label{a:4d}
$0<\int_{\Omega}\norm{\LW}^2\mu(d\omega)\leq 1$.
\end{enumerate}
\end{assumption}

The main purpose of this section is to study the following objects
which mix families of monotone and linear operators.

\begin{definition}
\label{d:8}
Suppose that Assumptions~\ref{a:2} and \ref{a:4} are in force. 
The \emph{integral resolvent mixture} of
$(\AW)_{\omega\in\Omega}$ and $(\LW)_{\omega\in\Omega}$ is
\begin{equation}
\label{e:100}
\rmi(\LW,\AW)_{\omega\in\Omega}
=\brk3{\int_{\Omega}\brk!{\LW^*\circ J_{\AW}\circ\LW}
\mu(d\omega)}^{-1}-\Id_{\XS},
\end{equation}
and the \emph{integral resolvent comixture} of
$(\AW)_{\omega\in\Omega}$ and $(\LW)_{\omega\in\Omega}$ is
\begin{equation}
\label{e:101}
\rcm(\LW,\AW)_{\omega\in\Omega}
=\brk2{\rmi\brk!{\LW,\AW^{-1}}_{\omega\in\Omega}}^{-1}.
\end{equation}
\end{definition}

We start off with some properties of integrals of composite
Lipschitzian operators.

\begin{proposition}
\label{p:4}
Suppose that Assumption~\ref{a:1} is in force.
Let $\XS$ be a separable real Hilbert space,
let $\beta\colon\Omega\to\RPP$ be $\FF$-measurable and such that
$\esssup\beta<\pinf$, and for every $\omega\in\Omega$,
let $\TW\colon\HW\to\HW$ be $\beta(\omega)$-Lipschitzian and let
$\LW\colon\XS\to\HW$ be linear and bounded. Suppose that the
following are satisfied:
\begin{enumerate}[label={\normalfont[\Alph*]}]
\item
\label{p:4a}
For every $x\in\GG$, the mapping
$\omega\mapsto\TW x(\omega)$ lies in $\mathfrak{G}$.
\item
\label{p:4b}
There exists $z\in\GG$ such that the mapping
$\omega\mapsto\TW z(\omega)$ lies in $\GG$.
\item
\label{p:4c}
For every $\mathsf{x}\in\XS$, the mapping
$\mathfrak{e}_{\LS}\mathsf{x}\colon
\omega\mapsto\LW\mathsf{x}$ lies in $\mathfrak{G}$.
\item
\label{p:4d}
$\int_{\Omega}\norm{\LW}^2\mu(d\omega)<\pinf$.
\end{enumerate}
Set
\begin{equation}
\TS=\int_{\Omega}\brk!{\LW^*\circ\TW\circ\LW}\mu(d\omega)
\quad\text{and}\quad
\tau=\int_{\Omega}\norm{\LW}^2\beta(\omega)\mu(d\omega).
\end{equation}
Then the following hold:
\begin{enumerate}
\item
\label{p:4i}
$\TS\colon\XS\to\XS$ is well defined and $\tau$-Lipschitzian.
\item
\label{p:4ii}
Define $L\colon\XS\to\HH\colon\mathsf{x}\mapsto
\mathfrak{e}_{\LS}\mathsf{x}$ and
$T=\leftindex^{\mathfrak{G}}{\int}_{\Omega}^{\oplus}
\TW\mu(d\omega)$. Then $L$ is well defined, linear, and bounded
with $\norm{L}\leq\sqrt{\int_{\Omega}\norm{\LW}^2\mu(d\omega)}$,
and $\TS=L^*\circ T\circ L$.
\item
\label{p:4iii}
Suppose that, for every $\omega\in\Omega$,
$\TW$ is $1/\beta(\omega)$-cocoercive and $\LW\neq 0$.
Then the following are satisfied:
\begin{enumerate}
\item
\label{p:4iiia}
$\TS$ is $1/\tau$-cocoercive.
\item
\label{p:4iiib}
$\menge{\int_{\Omega}\LW^*(\TW x(\omega))\mu(d\omega)}{
x\in\GG}\subset\cran\TS$.
\item
\label{p:4iiic}
$\inte\menge{\int_{\Omega}\LW^*(\TW x(\omega))\mu(d\omega)}{
x\in\GG}\subset\ran\TS$.
\end{enumerate}
\end{enumerate}
\end{proposition}
\begin{proof}
Observe that, by \cite[Proposition~3.12(i)]{Cana24},
the function $\Omega\to\RR\colon\omega\mapsto\norm{\LW}$
is $\FF$-measurable and, by \ref{p:4d},
\begin{equation}
\tau
\leq(\esssup\beta)\int_{\Omega}\norm{\LW}^2\mu(d\omega)
<\pinf.
\end{equation}
We set $(\forall\omega\in\Omega)$ $\RW=\LW^*\circ\TW\circ\LW$.

\ref{p:4i}:
Let $\mathsf{x}\in\XS$. It results from
\cite[Proposition~3.12(ii)]{Cana24} that the mapping
$\omega\mapsto\LW\mathsf{x}$ lies in $\GG$. In turn, \ref{p:4a}
ensures that the mapping $\omega\mapsto\TW(\LW\mathsf{x})$ lies in
$\mathfrak{G}$. Therefore, we deduce from \ref{p:4c} and
\cite[Lemma~2.2(i)]{Cana24} that, for every $\mathsf{y}\in\XS$,
the function $\Omega\to\RR\colon\omega\mapsto
\scal{\mathsf{y}}{\RW\mathsf{x}}_{\XS}
=\scal{\LW\mathsf{y}}{\TW(\LW\mathsf{x})}_{\HW}$ is
$\FF$-measurable. Thus, since $(\Omega,\FF,\mu)$ is a complete
$\sigma$-finite measure space and $\XS$ is separable, we infer
from \cite[Th\'eor\`eme~5.6.24]{Sch93b} that the mapping
$\Omega\to\XS\colon\omega\mapsto\RW\mathsf{x}$ is
$(\FF,\BE_{\XS})$-measurable. Next, since $\esssup\beta<\pinf$, it
follows from \ref{p:4a}, \ref{p:4b}, and
\cite[Proposition~3.4(i)]{Cana24} that, for every $x\in\GG$,
the mapping $\omega\mapsto\TW x(\omega)$ lies in $\GG$;
in particular,
$\int_{\Omega}\norm{\TW\mathsf{0}}_{\HW}^2\mu(d\omega)<\pinf$.
Hence, because
\begin{align}
\label{e:906w}
(\forall\mathsf{y}\in\XS)(\forall\omega\in\Omega)\quad
\norm{\RW\mathsf{x}-\RW\mathsf{y}}_{\XS}
&\leq\norm{\LW}\,
\norm{\TW(\LW\mathsf{x})-\TW(\LW\mathsf{y})}_{\HW}
\nonumber\\
&\leq\norm{\LW}\beta(\omega)
\norm{\LW\mathsf{x}-\LW\mathsf{y}}_{\HW}
\nonumber\\
&\leq\norm{\LW}^2\beta(\omega)\norm{\mathsf{x}-\mathsf{y}}_{\XS},
\end{align}
we derive from the triangle and Cauchy--Schwarz inequalities that
\begin{align}
\int_{\Omega}\norm{\RW\mathsf{x}}_{\XS}\,\mu(d\omega)
&\leq\int_{\Omega}\norm{\RW\mathsf{x}-\RW\mathsf{0}}_{\XS}\,
\mu(d\omega)+
\int_{\Omega}\norm{\RW\mathsf{0}}_{\XS}\,\mu(d\omega)
\nonumber\\
&\leq\tau\norm{\mathsf{x}}_{\XS}+
\int_{\Omega}\norm{\LW}\,\norm{\TW\mathsf{0}}_{\HW}\mu(d\omega)
\nonumber\\
&\leq\tau\norm{\mathsf{x}}_{\XS}+
\sqrt{\int_{\Omega}\norm{\LW}^2\mu(d\omega)}
\sqrt{\int_{\Omega}\norm{\TW\mathsf{0}}_{\HW}^2\mu(d\omega)}
\nonumber\\
&<\pinf.
\end{align}
Thus, \cite[Th\'eor\`eme~5.7.21]{Sch93b} implies that
$\TS\colon\XS\to\XS$ is well defined. Moreover, by virtue of
\eqref{e:906w} and Lemma~\ref{l:5}\ref{l:5i},
$\TS$ is $\tau$-Lipschitzian.

\ref{p:4ii}:
Thanks to \cite[Items (ii) and (v) in Proposition~3.12]{Cana24},
$L\colon\XS\to\HH$ is a well-defined bounded linear operator
with adjoint
\begin{equation}
\label{e:vdn1}
L^*\colon\HH\to\XS\colon
x^*\mapsto\int_{\Omega}\LW^*x^*(\omega)\mu(d\omega)
\end{equation}
and $\norm{L}\leq\sqrt{\int_{\Omega}\norm{\LW}^2\mu(d\omega)}$.
On the other hand, \cite[Proposition~3.4(i)]{Cana24}
asserts that $T\colon\HH\to\HH$ and that, for every $x\in\HH$,
a representative of $Tx$ in $\GG$ is the mapping
$\omega\mapsto\TW x(\omega)$. Altogether, for every
$\mathsf{x}\in\XS$, because $\omega\mapsto\LW\mathsf{x}$ is a
representative of $L\mathsf{x}$ in $\GG$, we deduce that
\begin{equation}
L^*\brk!{T(L\mathsf{x})}
=\int_{\Omega}\LW^*\brk!{\TW(\LW\mathsf{x})}\mu(d\omega)
=\TS\mathsf{x},
\end{equation}
as announced.

\ref{p:4iiia}:
Take $\mathsf{x}\in\XS$ and $\mathsf{y}\in\XS$.
Define an $\FF$-measurable function on $\Omega$ by
$\alpha\colon\Omega\to\RPP\colon\omega\mapsto
\norm{\LW}^2\beta(\omega)/\tau$
and a probability measure $\PP$ on $\FF$ by
$\PP\colon\Xi\mapsto\int_{\Xi}\alpha(\omega)\mu(d\omega)$.
Then we derive from items \ref{l:5ii} and \ref{l:5iii} of
Lemma~\ref{l:5} together with
\cite[Th\'eor\`eme~5.10.13]{Sch93b} that
\begin{align}
\scal{\mathsf{x}-\mathsf{y}}{\TS\mathsf{x}-\TS\mathsf{y}}_{\XS}
&=\int_{\Omega}
\scal!{\mathsf{x}-\mathsf{y}}{
\LW^*\brk!{\TW(\LW\mathsf{x})}-
\LW^*\brk!{\TW(\LW\mathsf{y})}}_{\XS}\,\mu(d\omega)
\nonumber\\
&=\int_{\Omega}\scal!{\LW\mathsf{x}-\LW\mathsf{y}}{
\TW(\LW\mathsf{x})-\TW(\LW\mathsf{y})}_{\HW}\mu(d\omega)
\nonumber\\
&\geq\int_{\Omega}\frac{1}{\beta(\omega)}
\norm!{\TW(\LW\mathsf{x})-\TW(\LW\mathsf{y})}_{\HW}^2\mu(d\omega)
\nonumber\\
&\geq\int_{\Omega}\frac{1}{\norm{\LW}^2\beta(\omega)}\norm!{
\LW^*\brk!{\TW(\LW\mathsf{x})}-
\LW^*\brk!{\TW(\LW\mathsf{y})}}_{\XS}^2\,\mu(d\omega)
\nonumber\\
&=\frac{1}{\tau}\int_{\Omega}\norm2{\frac{1}{\alpha(\omega)}
(\RW\mathsf{x}-\RW\mathsf{y})}_{\XS}^2\PP(d\omega)
\nonumber\\
&\geq\frac{1}{\tau}\norm3{\int_{\Omega}\frac{1}{\alpha(\omega)}
(\RW\mathsf{x}-\RW\mathsf{y})\PP(d\omega)}_{\XS}^2
\nonumber\\
&=\frac{1}{\tau}\norm3{\int_{\Omega}\brk!{
\RW\mathsf{x}-\RW\mathsf{y}}\mu(d\omega)}_{\XS}^2
\nonumber\\
&=\frac{1}{\tau}\norm{\TS\mathsf{x}-\TS\mathsf{y}}_{\XS}^2.
\end{align}

\ref{p:4iiib} and \ref{p:4iiic}:
Define $L$ and $T$ as in \ref{p:4ii}, and recall that
$\TS=L^*\circ T\circ L$. In the light of
\cite[Corollary~20.28]{Livre1}, \ref{p:4i} and \ref{p:4iiia} imply
that $\TS$ is maximally monotone. At the same time, we deduce from
\cite[Proposition~3.4(ii)]{Cana24} that $T\colon\HH\to\HH$ is
cocoercive. Thus, it results from \eqref{e:vdn1},
\cite[Example~25.20(i)]{Livre1}, and \cite[Theorem~5]{Penn01} that
\begin{equation}
\Menge3{\int_{\Omega}\LW^*\brk!{\TW x(\omega)}\mu(d\omega)}{
x\in\GG}
=L^*(\ran T)
\subset\cran\brk!{L^*\circ T\circ L}
=\cran\TS
\end{equation}
and that
\begin{equation}
\inte\Menge3{\int_{\Omega}\LW^*\brk!{\TW x(\omega)}\mu(d\omega)}{
x\in\GG}
=\inte L^*(\ran T)
\subset\ran\brk!{L^*\circ T\circ L}
=\ran\TS,
\end{equation}
which completes the proof.
\end{proof}

The main properties of integral resolvent mixtures can now be laid
out.

\begin{theorem}
\label{t:23}
Suppose that Assumptions~\ref{a:2} and \ref{a:4} are in force. 
Set
\begin{equation}
\mathsf{W}=\rmi(\LW,\AW)_{\omega\in\Omega}
\quad\text{and}\quad
\mathsf{C}=\rcm(\LW,\AW)_{\omega\in\Omega}.
\end{equation}
Then the following hold: 
\begin{enumerate}
\item
\label{t:23i}
$\mathsf{W}^{-1}=\rcm(\LW,\AW^{-1})_{\omega\in\Omega}$ and
$\mathsf{C}^{-1}=\rmi(\LW,\AW^{-1})_{\omega\in\Omega}$.
\item
\label{t:23ii}
$\mathsf{W}$ and $\mathsf{C}$ are maximally monotone.
\item
\label{t:23iii}
$\mathsf{C}=(\Id_{\XS}+\int_{\Omega}
(\LW^*\circ J_{\AW}\circ\LW-\LW^*\circ\LW)\mu(d\omega))^{-1}
-\Id_{\XS}$.
\item
\label{t:23iv}
Suppose that $\mu$ is a probability measure and that,
for every $\omega\in\Omega$, $\LW$ is an isometry. Then
$\mathsf{W}=\mathsf{C}$.
\item
\label{t:23v}
$J_{\mathsf{W}}=\int_{\Omega}
(\LW^*\circ J_{\AW}\circ\LW)\mu(d\omega)$.
\item
\label{t:23vi}
$J_{\mathsf{C}}=\Id_{\XS}+\int_\Omega
(\LW^*\circ J_{\AW}\circ\LW-\LW^*\circ\LW)\mu(d\omega)$.
\item
\label{t:23vii}
$\mathsf{W}\infconv\Id_{\XS}=
\Id_{\XS}-\int_{\Omega}(\LW^*\circ(\AW^{-1}\infconv\Id_{\HW})
\circ\LW)\mu(d\omega)$.
\item
\label{t:23viii}
$\mathsf{C}\infconv\Id_{\XS}=
\int_{\Omega}(\LW^*\circ
(\AW\infconv\Id_{\HW})\circ\LW)\mu(d\omega)$.
\item
\label{t:23viiibis}
$\zer\mathsf{C}=\zer\int_{\Omega}(\LW^*\circ
(\AW\infconv\Id_{\HW})\circ\LW)\mu(d\omega)$.
\item
\label{t:23ix}
$\cdom\mathsf{W}
=\overline{\menge{\int_{\Omega}\LW^*x^*(\omega)\mu(d\omega)}{
x^*\in\HH\,\,\text{and}\,\,(\forallmu\omega\in\Omega)\,\,
x^*(\omega)\in\dom\AW}}$.
\item
\label{t:23x}
$\cran\mathsf{C}
=\overline{\menge{\int_{\Omega}\LW^*x^*(\omega)\mu(d\omega)}{
x^*\in\HH\,\,\text{and}\,\,(\forallmu\omega\in\Omega)\,\,
x^*(\omega)\in\ran\AW}}$.
\item
\label{t:23x+}
$\intdom\mathsf{W}=\inte\menge{\int_{\Omega}\LW^*(J_{\AW}x(\omega))
\mu(d\omega)}{x\in\HH}$.
\item
\label{t:23x++}
$\inte\ran\mathsf{C}=\inte
\menge{\int_{\Omega}\LW^*(J_{\AW^{-1}}x(\omega))
\mu(d\omega)}{x\in\HH}$.
\item
\label{t:23xi}
Suppose that, for every $\omega\in\Omega$, $\AW$ is nonexpansive
with $\dom\AW=\HW$. Then $\mathsf{C}$ is nonexpansive.
\item
\label{t:23xii}
Let $\tau\in\RPP$, set $\delta=(\tau+1)/\int_{\Omega}
\norm{\LW}^2\mu(d\omega)-1$, and suppose that, for every
$\omega\in\Omega$, $\AW$ is $\tau$-cocoercive with $\dom\AW=\HW$.
Then $\mathsf{C}$ is $\delta$-cocoercive.
\end{enumerate}
\end{theorem}
\begin{proof}
Set
\begin{equation}
A=\leftindex^{\mathfrak{G}}{\int}_{\Omega}^{\oplus}
\AW\mu(d\omega).
\end{equation}
Then \cite[Theorem~3.8(i)]{Cana24} states that $A$ is maximally
monotone, and \cite[Theorem~3.8(ii)(a)]{Cana24} asserts that,
for every $x\in\GG$, the mapping $\omega\mapsto
J_{\AW}x(\omega)$ lies in $\GG$ and
\begin{equation}
J_A=\leftindex^{\mathfrak{G}}{\Int}_{\Omega}^{\oplus}
J_{\AW}\mu(d\omega).
\end{equation}
Therefore, by Assumption~\ref{a:4}\ref{a:4d} and
items \ref{p:4i} and \ref{p:4iiia} of Proposition~\ref{p:4},
\begin{equation}
\int_{\Omega}\brk!{\LW^*\circ J_{\AW}\circ\LW}\mu(d\omega)
\colon\XS\to\XS\,\,\text{is a well-defined firmly nonexpansive
operator},
\end{equation}
which confirms that $\mathsf{W}$ is well defined.
Additionally, it follows from Proposition~\ref{p:4}\ref{p:4ii} and
Assumption~\ref{a:4}\ref{a:4d} that
\begin{equation}
\label{e:rsmp}
L\colon\XS\to\HH\colon\mathsf{x}\mapsto
\mathfrak{e}_{\LS}\mathsf{x}\,\,\text{is a well-defined
bounded linear operator such that}\,\,\norm{L}\leq 1,
\end{equation}
and
\begin{equation}
\label{e:py6w}
(\forall\mathsf{x}\in\XS)\quad
L^*\brk!{J_A(L\mathsf{x})}
=\int_{\Omega}\LW^*\brk!{J_{\AW}(\LW\mathsf{x})}\mu(d\omega).
\end{equation}
Moreover, the adjoint of $L$ is given by
\cite[Proposition~3.12(v)]{Cana24}
\begin{equation}
\label{e:dvn2}
L^*\colon\HH\to\XS\colon
x^*\mapsto\int_{\Omega}\LW^*x^*(\omega)\mu(d\omega).
\end{equation}
Likewise, appealing to \cite[Proposition~3.7]{Cana24},
we deduce that $\mathsf{C}$ is well defined and
\begin{equation}
\label{e:py6c}
(\forall\mathsf{x}\in\XS)\quad
L^*\brk!{J_{A^{-1}}(L\mathsf{x})}
=\int_{\Omega}\LW^*\brk!{J_{\AW^{-1}}(\LW\mathsf{x})}\mu(d\omega).
\end{equation}
Hence, by virtue of Definition~\ref{d:s1},
\begin{equation}
\label{e:kvfp}
\mathsf{W}=\proxc{L}{A}\quad\text{and}\quad
\mathsf{C}=\proxcc{L}{A}.
\end{equation}

\ref{t:23i}:
A consequence of \eqref{e:100} and \eqref{e:101}.

\ref{t:23ii}:
In the light of \cite[Theorem~4.5(i)--(ii)]{Svva23},
the claim follows from \eqref{e:rsmp} and \eqref{e:kvfp}.

\ref{t:23iii}:
A consequence of \eqref{e:py6w}, \eqref{e:dvn2}, \eqref{e:kvfp},
and \cite[Proposition~4.1(ii)]{Svva23}.

\ref{t:23iv}:
By \eqref{e:y9d2} and \eqref{e:rsmp},
\begin{equation}
(\forall\mathsf{x}\in\XS)\quad
\norm{L\mathsf{x}}_{\HH}^2
=\int_{\Omega}\norm{\LW\mathsf{x}}_{\HW}^2\mu(d\omega)
=\int_{\Omega}\norm{\mathsf{x}}_{\XS}^2\,\mu(d\omega)
=\mu(\Omega)\norm{\mathsf{x}}_{\XS}^2
=\norm{\mathsf{x}}_{\XS}^2,
\end{equation}
which shows that $L$ is an isometry. Consequently, the conclusion
follows from \eqref{e:kvfp} and
\cite[Proposition~4.1(iii)]{Svva23}.

\ref{t:23v}:
An immediate consequence of \eqref{e:100}.

\ref{t:23vi}:
An immediate consequence of \ref{t:23iii}.

\ref{t:23vii}:
This follows from \eqref{e:kvfp},
\cite[Proposition~4.1(xiv)]{Svva23}, and \eqref{e:py6w}.

\ref{t:23viii}:
This follows from \eqref{e:kvfp},
\cite[Proposition~4.1(xv)]{Svva23}, and \eqref{e:py6c}.

\ref{t:23viiibis}:
Use \ref{t:23ii}, \ref{t:23viii}, and
\cite[Proposition~23.38]{Livre1}.

\ref{t:23ix}:
Set $U=\menge{x\in\HH}{(\forallmu\omega\in\Omega)\,\,
x(\omega)\in\dom\AW}$. Then $\overline{U}=\cdom A$
\cite[Theorem~3.8(iii)]{Cana24}.
Hence, \cite[Theorem~4.5(vi)]{Svva23} implies that
\begin{equation}
\cdom\mathsf{W}
=\overline{L^*(\dom A)}
=\overline{L^*\brk!{\cdom A}}
=\overline{L^*(\overline{U})}
=\overline{L^*(U)}.
\end{equation}
This and \eqref{e:dvn2} yield the desired identity.

\ref{t:23x}:
Combine \eqref{e:101}, \ref{t:23ix}, and the fact that
$(\forall\omega\in\Omega)$ $\dom\AW^{-1}=\ran\AW$.

\ref{t:23x+}:
Use \eqref{e:100} and Proposition~\ref{p:4}\ref{p:4iiic}.

\ref{t:23x++}:
A consequence of \eqref{e:101} and \ref{t:23x+}.

\ref{t:23xi}:
It follows from \cite[Theorem~3.8(v)(a)]{Cana24} that
\begin{equation}
\text{for every $x\in\GG$, the mapping
$\omega\mapsto\proj_{\AW x(\omega)}\mathsf{0}=\AW x(\omega)$ lies
in $\mathfrak{G}$}.
\end{equation}
Hence, \cite[Proposition~3.4(i)]{Cana24}
implies that $A$ is nonexpansive with $\dom A=\HH$.
Hence, it follows from \cite[Proposition~4.9]{Svva23} and
\eqref{e:kvfp} that $\mathsf{C}$ is nonexpansive.

\ref{t:23xii}:
We argue as in \ref{t:23xi} to deduce that
$A\colon\HH\to\HH$ is $\tau$-cocoercive.
On the other hand, \eqref{e:rsmp} ensures that
$\norm{L}<\sqrt{\tau+1}$. Thus, it follows from
\eqref{e:kvfp}, \cite[Proposition~4.8]{Svva23},
and Proposition~\ref{p:4}\ref{p:4ii} that
$\mathsf{C}=\proxcc{L}{A}$ is cocoercive with constant
$(\tau+1)\norm{L}^{-2}-1\geq\delta$.
\end{proof}

\begin{remark}
The motivation for calling $\rmi(\LW,\AW)_{\omega\in\Omega}$ an
integral resolvent mixture comes from
Theorem~\ref{t:23}\ref{t:23v}.
\end{remark}

Let us provide some examples of integral resolvent mixtures.

\begin{example}
\label{ex:100}
Consider the setting of Example~\ref{ex:1i+}. Then \eqref{e:100}
becomes
\begin{equation}
\label{e:62}
\rmi(\LS_k,\AS_k)_{1\leq k\leq p}
=\brk4{\sum_{k=1}^p\alpha_k\LS_k^*\circ
J_{\AS_k}\circ\LS_k}^{-1}-\Id_{\XS},
\end{equation}
which is the \emph{resolvent mixture} introduced in
\cite[Example~3.4]{Svva23}.
\end{example}

\begin{example}
\label{ex:2005}
Let $(\Omega,\FF,\mu)$ be a complete $\sigma$-finite measure space
and let $(\phi_{\omega})_{\omega\in\Omega}$
be a family in $\Gamma_0(\RR)$ such that the function
$\Omega\times\RR\to\RX\colon(\omega,\mathsf{x})\mapsto
\phi_{\omega}(\mathsf{x})$ is $\FF\otimes\BE_{\RR}$-measurable
and $(\forall\omega\in\Omega)$
$\phi_{\omega}\geq\phi_{\omega}(0)=0$.
Further, let $\XS$ be a separable real Hilbert space and let 
$e\in L^2\brk{\Omega,\FF,\mu;\XS}$ be such that 
$0<\int_\Omega\norm{e(\omega)}_{\XS}^2\,\mu(d\omega)\leq 1$. Set
\begin{equation}
(\forall\omega\in\Omega)\quad
\AW=\partial\phi_{\omega}\quad\text{and}\quad
\LW=\scal{\Cdot}{e(\omega)}_{\XS}.
\end{equation}
Then
\begin{equation}
\label{e:68}
\rmi(\LW,\AW)_{\omega\in\Omega}
=\brk3{\int_{\Omega}\brk!{\prox_{\phi_{\omega}}
\scal{\Cdot}{e(\omega)}_{\XS}}e(\omega)\mu(d\omega)}^{-1}
-\Id_{\XS}.
\end{equation}
For instance, suppose that, for every $\omega\in\Omega$, 
$\phi_{\omega}$ is the support function of a closed interval
$\mathsf{C}_{\omega}$ in $\RR$ containing $0$, with
$\delta_{\omega}=\inf\mathsf{C}_{\omega}$ and
$\rho_{\omega}=\sup\mathsf{C}_{\omega}$. Now set
\begin{equation}
\mathsf{W}=\rmi(\LW,\AW)_{\omega\in\Omega}
\quad\text{and}\quad
(\forall\mathsf{x}\in\XS)\;\;
\begin{cases}
\Omega^-(\mathsf{x})=\menge{\omega\in\Omega}
{\scal{\mathsf{x}}{e(\omega)}_{\XS}>\rho_{\omega}}\\
\Omega_-(\mathsf{x})=\menge{\omega\in\Omega}
{\scal{\mathsf{x}}{e(\omega)}_{\XS}<\delta_{\omega}}.
\end{cases}
\end{equation}
Then
\begin{equation}
\label{e:h1qe}
(\forall\mathsf{x}\in\XS)\quad
J_{\mathsf{W}}\mathsf{x}=
\int_{\Omega^-(\mathsf{x})}\brk!{\scal{\mathsf{x}}{e(\omega)}_{\XS}
-\rho_{\omega}}e(\omega)\mu(d\omega)+
\int_{\Omega_-(\mathsf{x})}\brk!{\scal{\mathsf{x}}{e(\omega)}_{\XS}
-\delta_{\omega}}e(\omega)\mu(d\omega).
\end{equation}
This process provides a representation of $\mathsf{x}$ which
eliminates the contributions of the coefficients 
$\scal{\mathsf{x}}{e(\omega)}_{\XS}\in
\intv{\delta_{\omega}}{\rho_{\omega}}$. 
For instance, in the context of Example~\ref{ex:1ii}, if
$(e(k))_{k\in\NN}$ is an orthonormal basis and
$\mathsf{C}_k=\intv{-\rho_k}{\rho_k}$,
then $J_{\mathsf{W}}$ is known as a soft-thresholder and it has
been used extensively in data analysis \cite{Smms05,Daub04}.
\end{example}
\begin{proof}
Let $\mathfrak{G}=\menge{x\colon\Omega\to\RR}{
x\,\,\text{is $\FF$-measurable}}$ and,
for every $\omega\in\Omega$, let $\HW=\RR$.
Then, in view of Example~\ref{ex:1iii},
Assumption~\ref{a:1} is satisfied and
\begin{equation}
\HH
=\leftindex^{\mathfrak{G}}{\int}_{\Omega}^{\oplus}
\HW\mu(d\omega)
=L^2\brk!{\Omega,\FF,\mu;\RR}.
\end{equation}
Since $(\Omega,\FF,\mu)$ is complete,
we deduce from \cite[Corollary~14.34 and Exercise~14.38]{Rock09}
that, for every $x\in\mathfrak{G}$, the function
$\Omega\to\RR\colon\omega\mapsto\prox_{\phi_{\omega}}x(\omega)$
lies in $\mathfrak{G}$. Additionally,
for every $\omega\in\Omega$, since $0\in\Argmin\phi_{\omega}$,
we get $0\in\AW 0$ and $J_{\AW}0=\prox_{\phi_{\omega}}0=0$.
Hence, the family $(\AW)_{\omega\in\Omega}$ satisfies
Assumption~\ref{a:2}. Next, since $e\colon\Omega\to\XS$
is $(\FF,\BE_{\XS})$-measurable, we deduce that,
for every $\mathsf{x}\in\XS$, the mapping
$\Omega\to\RR\colon\omega\mapsto
\scal{\mathsf{x}}{e(\omega)}_{\XS}=\LW\mathsf{x}$
lies in $\mathfrak{G}$. Further,
\begin{equation}
(\forall\omega\in\Omega)\quad
\LW^*\colon\RR\to\XS\colon\mathsf{x}\mapsto\mathsf{x}e(\omega)
\end{equation}
and
\begin{equation}
\int_{\Omega}\norm{\LW}^2\mu(d\omega)
=\int_{\Omega}\norm{e(\omega)}_{\XS}^2\,\mu(d\omega)
=\norm{e}_{\HH}^2
\in\intv[l]{0}{1}.
\end{equation}
This confirms that Assumption~\ref{a:4} is satisfied.
Therefore, we obtain \eqref{e:68} by invoking \eqref{e:100}.
Next, let us establish \eqref{e:h1qe}.
Take $\mathsf{x}\in\XS$. Thanks to the $\FF$-measurability of the
function $\Omega\to\RR\colon\omega\mapsto
\prox_{\phi_{\omega}}\scal{\mathsf{x}}{e(\omega)}_{\XS}$, we
obtain
\begin{align}
\Omega^-(\mathsf{x})
&=\menge{\omega\in\Omega}{\scal{\mathsf{x}}{e(\omega)}_{\XS}-
\proj_{\mathsf{C}_{\omega}}\scal{\mathsf{x}}{e(\omega)}_{\XS}>0}
\nonumber\\
&=\menge{\omega\in\Omega}{\prox_{\phi_{\omega}}
\scal{\mathsf{x}}{e(\omega)}_{\XS}>0}
\nonumber\\
&\in\FF.
\end{align}
Likewise, $\Omega_-(\mathsf{x})\in\FF$.
On the other hand, by \cite[Example~24.34]{Livre1},
\begin{equation}
(\forall\omega\in\Omega)\quad
\prox_{\phi_{\omega}}\colon\RR\to\RR\colon\mathsf{x}\mapsto
\begin{cases}
\mathsf{x}-\rho_{\omega},&\text{if}\,\,\mathsf{x}>\rho_{\omega};\\
0,&\text{if}\,\,\mathsf{x}\in\mathsf{C}_{\omega};\\
\mathsf{x}-\delta_{\omega},&\text{if}\,\,
\mathsf{x}<\delta_{\omega}.
\end{cases}
\end{equation}
Therefore, we obtain \eqref{e:h1qe} by using
Theorem~\ref{t:23}\ref{t:23v} and the fact that
$(\forall\omega\in\Omega)$
$J_{\AW}=\prox_{\phi_{\omega}}$.
\end{proof}

Next, we define the resolvent expectation of a family of maximally
monotone operators.

\begin{definition}
\label{d:101}
Let $(\Omega,\FF,\PP)$ be a complete probability space,
let $\HS$ be a separable real Hilbert space,
and let $(\AW)_{\omega\in\Omega}$ be a family of maximally
monotone operators from $\HS$ to $2^{\HS}$.
Suppose that, for every $\mathsf{x}\in\HS$, the mapping
$\Omega\to\HS\colon\omega\mapsto J_{\AW}\mathsf{x}$
is $(\FF,\BE_{\HS})$-measurable and that
$\int_{\Omega}\norm{J_{\AW}\mathsf{0}}_{\HS}^2\,\mu(d\omega)
<\pinf$. Using the notation \eqref{e:vk}, the \emph{resolvent
expectation} of the family $(\AW)_{\omega\in\Omega}$ is 
\begin{equation}
\label{e:fr}
\PE(\AW)_{\omega\in\Omega}=\brk!{\EE
(J_{\AW})_{\omega\in\Omega}}^{-1}-\Id_{\HS}.
\end{equation}
\end{definition}

\begin{example}
Consider the measure space $(\Omega,\FF,\mu)$ of
Example~\ref{ex:1i+} with the additional assumption that
$\sum_{k=1}^p\alpha_k=1$. Let $\HS$ be a separable real Hilbert
space and let $(\AS_k)_{1\leq k\leq p}$ be
maximally monotone operators from $\HS$ to $2^{\HS}$. Then
\eqref{e:fr} becomes
\begin{equation}
\label{e:106}
\PE(\AS_k)_{1\leq k\leq p}=\brk3{\sum_{k=1}^p\alpha_k
J_{\AS_k}}^{-1}-\Id_{\HS},
\end{equation}
which is the \emph{resolvent average} studied in \cite{Baus16}.
\end{example}

Let us relate resolvent expectations to integral resolvent
mixtures.

\begin{proposition}
\label{p:101}
Consider the setting of Example~\ref{ex:1iii} with the additional
assumption that $\mu$ is a probability measure.
Let $(\AW)_{\omega\in\Omega}$ be a family of maximally
monotone operators from $\HS$ to $2^{\HS}$ such that,
for every $\mathsf{x}\in\HS$, the mapping
$\Omega\to\HS\colon\omega\mapsto J_{\AW}\mathsf{x}$
is $(\FF,\BE_{\HS})$-measurable and that
$\int_{\Omega}\norm{J_{\AW}\mathsf{0}}_{\HS}^2\,\mu(d\omega)
<\pinf$. Then
\begin{equation}
\label{e:105}
\PE(\AW)_{\omega\in\Omega}
=\rmi(\Id_{\HS},\AW)_{\omega\in\Omega}
=\rcm(\Id_{\HS},\AW)_{\omega\in\Omega}.
\end{equation}
\end{proposition}
\begin{proof}
Appealing to \cite[Lemma~III.14]{Cast77} and
the continuity of the operators $(J_{\AW})_{\omega\in\Omega}$,
we infer that the mapping
$\Omega\times\HS\to\HS\colon(\omega,\mathsf{x})\mapsto
J_{\AW}\mathsf{x}$ is
$(\FF\otimes\BE_{\HS},\BE_{\HS})$-measurable. Thus, for every
$x\in\GG$, the mapping $\Omega\to\HS\colon\omega\mapsto
J_{\AW}x(\omega)$ is $(\FF,\BE_{\HS})$-measurable, i.e., it lies
in $\mathfrak{G}$. On the other hand,
letting $A=\leftindex^{\mathfrak{G}}{\int}_{\Omega}^{\oplus}
\AW\mu(d\omega)$ and $r\colon\Omega\to\HS\colon\omega\mapsto
J_{\AW}\mathsf{0}$ yields ${-}r\in Ar$, which implies that
$\dom A\neq\emp$. Hence, it follows from \eqref{e:fr} and
\eqref{e:100} that $\PE(\AW)_{\omega\in\Omega}=
\rmi(\Id_{\HS},\AW)_{\omega\in\Omega}$,
while the identity
$\rmi(\Id_{\HS},\AW)_{\omega\in\Omega}
=\rcm(\Id_{\HS},\AW)_{\omega\in\Omega}$ follows from
Theorem~\ref{t:23}\ref{t:23iv}.
\end{proof}

By specializing Theorem~\ref{t:23} to the scenario of
Proposition~\ref{p:101}, we obtain at once the following properties
of the resolvent expectation and, in particular, those of the
resolvent average of finitely many operators studied in
\cite{Baus16}.

\begin{corollary}
\label{c:23}
Consider the setting of Definition~\ref{d:101}.
Then the following hold: 
\begin{enumerate}
\item
\label{c:23i}
$(\PE(\AW)_{\omega\in\Omega})^{-1}
=\PE(\AW^{-1})_{\omega\in\Omega}$.
\item
\label{c:23ii}
$\PE(\AW)_{\omega\in\Omega}$ is maximally monotone.
\item
\label{c:23v}
$J_{\PE(\AW)_{\omega\in\Omega}}=\EE(J_{\AW})_{\omega\in\Omega}$.
\item
\label{c:23vii}
$(\PE(\AW)_{\omega\in\Omega})\infconv\Id_{\HS}=
\EE(\AW\infconv\Id_{\HS})_{\omega\in\Omega}$.
\item
\label{c:23ix}
$\cdom\PE(\AW)_{\omega\in\Omega}
=\overline{\menge{\EE x^*}{
x^*\in L^2\brk{\Omega,\FF,\PP;\HS}\,\,\text{and}\,\,
(\forallmu\omega\in\Omega)\,\,x^*(\omega)\in\dom\AW}}$. 
\item
\label{c:23x}
$\cran\PE(\AW)_{\omega\in\Omega}=
\overline{\menge{\EE x^*}{
x^*\in L^2\brk{\Omega,\FF,\PP;\HS}\,\,\text{and}\,\,
(\forallmu\omega\in\Omega)\,\,x^*(\omega)\in\ran\AW}}$. 
\item
\label{c:23x+}
$\intdom\PE(\AW)_{\omega\in\Omega}
=\inte\menge{\EE(J_{\AW}x(\omega))_{\omega\in\Omega}}{
x\in L^2\brk{\Omega,\FF,\PP;\HS}}$.
\item
\label{c:23x++}
$\inte\ran\PE(\AW)_{\omega\in\Omega}=\inte
\menge{\EE(J_{\AW^{-1}}x(\omega))_{\omega\in\Omega}}{
x\in L^2\brk{\Omega,\FF,\PP;\HS}}$.
\item
\label{c:23xi}
Suppose that, for every $\omega\in\Omega$, 
$\AW$ is nonexpansive with $\dom\AW=\HW$. Then 
$\PE(\AW)_{\omega\in\Omega}$ is nonexpansive.
\item
\label{c:23xii}
Let $\tau\in\RPP$ and suppose that, for every
$\omega\in\Omega$, $\AW$ is $\tau$-cocoercive with $\dom\AW=\HW$.
Then $\PE(\AW)_{\omega\in\Omega}$ is $\tau$-cocoercive.
\end{enumerate}
\end{corollary}

\section{Integral proximal mixtures}
\label{sec:4}

The integral proximal mixture will be cast in the following
setting.

\begin{assumption}
\label{a:3}
Assumption~\ref{a:1} and the following are in force:
\begin{enumerate}[label={\normalfont[\Alph*]}]
\item
\label{a:3a}
For every $\omega\in\Omega$, $\fw\colon\HW\to\RX$
possesses a continuous affine minorant.
\item
\label{a:3b}
There exists $r\in\GG$ such that the function
$\omega\mapsto\fw(r(\omega))$ lies in
$\mathscr{L}^1(\Omega,\FF,\mu;\RR)$.
\item
\label{a:3c}
There exists $r^*\in\GG$ such that the function
$\omega\mapsto\fw^*(r^*(\omega))$ lies in
$\mathscr{L}^1(\Omega,\FF,\mu;\RR)$.
\item
\label{a:3d}
For every $x^*\in\GG$, the mapping
$\omega\mapsto\prox_{\fw^*}x^*(\omega)$ lies in $\mathfrak{G}$.
\end{enumerate}
\end{assumption}

\begin{definition}
\label{d:7}
Suppose that Assumptions~\ref{a:4} and \ref{a:3} are in force. 
The \emph{integral proximal mixture} of
$(\fw)_{\omega\in\Omega}$ and $(\LW)_{\omega\in\Omega}$ is
\begin{equation}
\label{e:120}
\rmi(\LW,\fw)_{\omega\in\Omega}
=\brk3{\int_{\Omega}\brk!{(\fw^*\infconv\qq_{\HW})\circ\LW}
\mu(d\omega)}^*-\qq_{\XS},
\end{equation}
and the \emph{integral proximal comixture} of
$(\fw)_{\omega\in\Omega}$ and
$(\LW)_{\omega\in\Omega}$ is
\begin{equation}
\label{e:121}
\rcm(\LW,\fw)_{\omega\in\Omega}
=\brk2{\rmi\brk!{\LW,\fw^*}_{\omega\in\Omega}}^*.
\end{equation}
\end{definition}

Item~\ref{t:25viii} below connects Definitions~\ref{d:8} and
\ref{d:7}.

\begin{theorem}
\label{t:25}
Suppose that Assumptions~\ref{a:4} and \ref{a:3} are in force.
Then the following hold:
\begin{enumerate}
\item
\label{t:25i}
$\rmi(\LW,\fw)_{\omega\in\Omega}\in\Gamma_0(\XS)$.
\item
\label{t:25ii}
$\rcm(\LW,\fw)_{\omega\in\Omega}\in\Gamma_0(\XS)$.
\item
\label{t:25iii}
Let $\mathsf{x}\in\XS$. Then 
\begin{multline*}
\hskip -9mm
\bigl(\rmi(\LW,\fw)_{\omega\in\Omega}\bigr)(\mathsf{x})\\
=\min\Menge3{\int_{\Omega}\fw^{**}\brk!{x(\omega)}\mu(d\omega)
+\qq_{\HH}(x)-\qq_{\XS}(\mathsf{x})}{x\in\HH\,\,\text{and}\,\,
\int_{\Omega}\LW^*x(\omega)\mu(d\omega)=\mathsf{x}}.
\end{multline*}
\item
\label{t:25iv}
$\cdom\rmi(\LW,\fw)_{\omega\in\Omega}=\overline{
\menge{\int_{\Omega}\LW^*x(\omega)\mu(d\omega)}{
x\in\HH\,\,\text{and}\,\,(\forallmu\omega\in\Omega)\,\,
x(\omega)\in\dom\fw^{**}}}$.
\item
\label{t:25v}
$(\rmi(\LW,\fw)_{\omega\in\Omega})^*
=\rcm(\LW,\fw^*)_{\omega\in\Omega}
=(\qq_{\XS}-\int_{\Omega}(\fw^*\infconv\qq_{\HW})
\circ\LW\,\mu(d\omega))^*-\qq_{\XS}$.
\item
\label{t:25vi}
$(\rcm(\LW,\fw)_{\omega\in\Omega})^*
=\rmi(\LW,\fw^*)_{\omega\in\Omega}$.
\item
\label{t:25vii}
$\rmi(\LW,\fw)_{\omega\in\Omega}\infconv\qq_{\XS}+
\rcm(\LW,\fw^*)_{\omega\in\Omega}\infconv\qq_{\XS}=\qq_{\XS}$.
\item
\label{t:25viii}
$\partial\rmi(\LW,\fw)_{\omega\in\Omega}=
\rmi(\LW,\partial\fw^{**})_{\omega\in\Omega}$.
\item
\label{t:25ix}
$\prox_{\rmi(\LW,\fw)_{\omega\in\Omega}}=
\int_{\Omega}\brk{\LW^*\circ\prox_{\fw^{**}}\circ\LW}
\mu(d\omega)$.
\item
\label{t:25x}
$\prox_{\rcm(\LW,\fw)_{\omega\in\Omega}}=
\Id_{\XS}-
\int_{\Omega}\brk{\LW^*\circ\prox_{\fw^*}\circ\LW}\mu(d\omega)$.
\item
\label{t:25xi}
$\rcm(\LW,\fw)_{\omega\in\Omega}\infconv\qq_{\XS}=
\int_{\Omega}(\fw^{**}\infconv\qq_{\HW})\circ\LW\,\mu(d\omega)$.
\item
\label{t:25xii}
$\Argmin\rcm(\LW,\fw)_{\omega\in\Omega}
=\Argmin
\int_{\Omega}(\fw^{**}\infconv\qq_{\HW})\circ\LW\,\mu(d\omega)$.
\item
\label{t:25xiii}
Suppose that $\mu$ is a probability measure and that,
for every $\omega\in\Omega$, $\LW$ is an isometry. Then
$\rmi(\LW,\fw)_{\omega\in\Omega}=\rcm(\LW,\fw)_{\omega\in\Omega}$.
\end{enumerate}
\end{theorem}
\begin{proof}
Set
\begin{equation}
\mathsf{g}=\rmi(\LW,\fw)_{\omega\in\Omega}\;\;\text{and}\;\;
\mathsf{h}=\rcm(\LW,\fw)_{\omega\in\Omega}.
\end{equation}
In the light of \cite[Propositions~13.12(ii) and
13.10(ii)]{Livre1}, we infer from
\ref{a:3a} and \ref{a:3b} of Assumption~\ref{a:3} that
$(\forall\omega\in\Omega)$ $\fw^*\in\Gamma_0(\HW)$.
Next, define $\varrho\colon\Omega\to\RR\colon\omega\mapsto
{-}\fw(r(\omega))$. Then, by Assumption~\ref{a:3}\ref{a:3b},
$\varrho\in\mathscr{L}^1(\Omega,\FF,\mu;\RR)$. Additionally,
$(\forall\omega\in\Omega)$
$\fw^*\geq\scal{\Cdot}{r(\omega)}+\varrho(\omega)$.
Hence, we conclude that the family
$(\fw^*)_{\omega\in\Omega}$ satisfies the assumptions of
\cite[Theorem~4.7]{Cana24} and therefore that the family
$(\partial\fw^{**})_{\omega\in\Omega}$
satisfies Assumption~\ref{a:2}.
Let us now check that the family
$(\fw^{*})_{\omega\in\Omega}$ satisfies Assumption~\ref{a:3} by
using the mapping $r$ to fulfill \ref{a:3c}. To this end, we need
to show that the function
$\varphi\colon\omega\mapsto\fw^{**}(r(\omega))$ lies in
$\mathscr{L}^1(\Omega,\FF,\mu;\RR)$. First, it follows from
\cite[Theorem~4.7(ix)]{Cana24} that $\varphi$ is $\FF$-measurable.
Further, since, for every $\omega\in\Omega$, $\fw^{**}\leq\fw$, 
Assumption~\ref{a:3}\ref{a:3b}
implies that $\varphi$ is majorized by an integrable function.
Finally, since 
\begin{equation}
(\forall\omega\in\Omega)\quad
\fw^{**}\geq\scal{\cdot}{r^*(\omega)}_{\HW}
-\fw^*\brk1{r^*(\omega)},
\end{equation}
Assumption~\ref{a:3}\ref{a:3c} implies that $\varphi$ is
minorized by an integrable function.
Next, we observe that \cite[Theorem~4.7(i)--(ii)]{Cana24} assert
that
\begin{equation}
g\colon\HH\to\RX\colon
x^*\mapsto\int_{\Omega}\fw^*\brk!{x^*(\omega)}\mu(d\omega)
\end{equation}
is a well-defined function in $\Gamma_0(\HH)$.
Moreover, by \cite[Theorem~4.7(viii)]{Cana24},
\begin{equation}
\label{e:u1ng}
g\infconv\qq_{\HH}\colon\HH\to\RR\colon x^*\mapsto
\int_{\Omega}(\fw^*\infconv\qq_{\HW})\brk!{x^*(\omega)}\mu(d\omega)
\end{equation}
and, by \cite[Theorem~4.7(ix)]{Cana24},
\begin{equation}
\label{e:fdvn}
g^*\colon\HH\to\RX\colon x\mapsto
\int_{\Omega}\fw^{**}\brk!{x(\omega)}\mu(d\omega).
\end{equation}
We also recall from \eqref{e:rsmp} that
\begin{equation}
\label{e:uvhz}
L\colon\XS\to\HH\colon\mathsf{x}\mapsto
\mathfrak{e}_{\LS}\mathsf{x}\,\,\text{is a well-defined
bounded linear operator with}\,\,\norm{L}\leq 1.
\end{equation}

\ref{t:25i}:
We deduce from \eqref{e:120},
Moreau's biconjugation theorem \cite[Corollary~13.38]{Livre1},
Definition~\ref{d:s2}, and
\cite[Example~3.6(ii)]{Svva23} that
\begin{equation}
\label{e:nn5j}
\rmi(\LW,\fw)_{\omega\in\Omega}
=\brk!{\brk{g\infconv\qq_{\HH}}\circ L}^*-\qq_{\XS}
=\proxc{L}{g^*}
\in\Gamma_0(\XS).
\end{equation}

\ref{t:25ii}:
It follows from \ref{t:25i}, Definition~\ref{d:s2}, and
\cite[Example~3.10(i)]{Svva23} that
\begin{equation}
\label{e:4cpy}
\mathsf{h}
=\brk2{\rmi\brk!{\LW,\fw^*}_{\omega\in\Omega}}^*
=\brk!{\proxc{L}{g^{**}}}^*
=\proxcc{L}{g^*}
\in\Gamma_0(\XS).
\end{equation}

\ref{t:25iii}:
We derive from \eqref{e:nn5j} and
\cite[Corollary~15.28(i) and Proposition~13.24(i)]{Livre1} that
\begin{align}
(\forall\mathsf{x}\in\XS)\quad
\mathsf{g}(\mathsf{x})
&=\min\menge{(g\infconv\qq_{\HH})^*(x)-\qq_{\XS}(\mathsf{x})}{
x\in\HH\,\,\text{and}\,\,L^*x=\mathsf{x}}
\nonumber\\
&=\min\menge{g^*(x)+\qq_{\HH}(x)-\qq_{\XS}(\mathsf{x})}{
x\in\HH\,\,\text{and}\,\,L^*x=\mathsf{x}}.
\end{align}
Thus, \eqref{e:fdvn} and \eqref{e:dvn2} yield the announced
identity.

\ref{t:25iv}:
Set $U=\menge{x\in\HH}{(\forallmu\omega\in\Omega)\,\,
x(\omega)\in\dom\fw^{**}}$. Then \cite[Theorem~4.7(v)]{Cana24}
states that $\cdom g^*=\overline{U}$. Thus, it results from
\eqref{e:nn5j} and \cite[Theorem~5.5(ii)]{Svva23} that
\begin{align}
\cdom\mathsf{g}
=\cdom(\proxc{L}{g^*})
=\overline{L^*(\dom g^*)}
=\overline{L^*\brk!{\cdom g^*}}
=\overline{L^*(\overline{U})}
=\overline{L^*(U)},
\end{align}
and the assertion follows from \eqref{e:dvn2}.

\ref{t:25v}:
It follows from \eqref{e:nn5j},
\cite[Proposition~5.3(iv)]{Svva23}, and \eqref{e:121} that
\begin{equation}
\mathsf{g}^*
=\brk!{\proxc{L}{g^*}}^*
=\proxcc{L}{g^{**}}
=\rcm(\LW,\fw^*)_{\omega\in\Omega}.
\end{equation}
At the same time, we derive from \eqref{e:nn5j},
\cite[Proposition~13.29]{Livre1}, and \eqref{e:u1ng} that
\begin{equation}
\mathsf{g}^*
=\brk!{\qq_{\XS}-\brk!{\brk{g\infconv\qq_{\HH}}\circ L}^{**}}^*
-\qq_{\XS}
=\brk3{\qq_{\XS}-\int_{\Omega}(\fw^*\infconv\qq_{\HW})
\circ\LW\,\mu(d\omega)}^*-\qq_{\XS}.
\end{equation}

\ref{t:25vi}:
Since $g\in\Gamma_0(\HH)$, we deduce from
\eqref{e:4cpy}, \cite[Proposition~5.3(v)]{Svva23}, Moreau's
biconjugation theorem, and \ref{t:25i} that
\begin{equation}
\mathsf{h}^*
=\brk!{\proxcc{L}{g^*}}^*
=\proxc{L}{g^{**}}
=\proxc{L}{g}
=\rmi(\LW,\fw^*)_{\omega\in\Omega}.
\end{equation}

\ref{t:25vii}:
Use \ref{t:25i}, \ref{t:25v}, and \cite[Theorem~14.3(i)]{Livre1}.

\ref{t:25viii}:
In view of \eqref{e:nn5j}, we derive from
\cite[Theorem~18.15]{Livre1},
\cite[Theorem~4.7(iv)]{Cana24}, and \eqref{e:dvn2} that
\begin{align}
\partial\rmi(\LW,\fw)_{\omega\in\Omega}
&=\brk2{\nabla\brk!{\brk{g\infconv\qq_{\HH}}\circ L}}^{-1}
-\Id_{\XS}
\nonumber\\
&=\brk2{\mathop{L^*}\circ
\brk!{\nabla\brk{g\infconv\qq_{\HH}}}\circ
\mathop{L}}^{-1}-\Id_{\XS}
\nonumber\\
&=\brk3{\int_{\Omega}\brk!{\LW^*\circ\prox_{\fw^{**}}\circ\LW}
\mu(d\omega)}^{-1}-\Id_{\XS}
\label{e:ebat}
\\
&=\rmi(\LW,\partial\fw^{**})_{\omega\in\Omega}.
\end{align}

\ref{t:25ix}:
Use \eqref{e:ebat} and \cite[Example~23.3]{Livre1}.

\ref{t:25x}:
By \cite[Proposition~13.16(iii)]{Livre1},
$(\forall\omega\in\Omega)$ $\fw^{***}=\fw^*$.
Hence, it results from \ref{t:25ii}, Moreau's decomposition
\cite[Theorem~14.3(ii)]{Livre1}, \ref{t:25vi}, and \ref{t:25ix}
that
\begin{align}
\prox_{\mathsf{h}}
&=\Id_{\XS}-\prox_{\mathsf{h}^*}
\nonumber\\
&=\Id_{\XS}-\prox_{\rmi(\LW,\fw^*)_{\omega\in\Omega}}
\nonumber\\
&=\Id_{\XS}-\int_{\Omega}\brk!{\LW^*\circ\prox_{\fw^{***}}\circ\LW}
\mu(d\omega)
\nonumber\\
&=\Id_{\XS}-\int_{\Omega}\brk!{\LW^*\circ\prox_{\fw^*}\circ\LW}
\mu(d\omega).
\end{align}

\ref{t:25xi}:
Because $g^*\in\Gamma_0(\HH)$, it results from
\eqref{e:4cpy}, \cite[Theorem~5.5(v)]{Svva23},
\cite[Theorem~4.7(viii)]{Cana24}, and \eqref{e:uvhz} that
\begin{equation}
\mathsf{h}\infconv\qq_{\XS}
=\brk!{\proxcc{L}{g^*}}\infconv\qq_{\XS}
=(g^*\infconv\qq_{\HH})\circ L
=\int_{\Omega}\brk2{\brk!{\fw^{**}\infconv\qq_{\HW}}\circ\LW}
\mu(d\omega).
\end{equation}

\ref{t:25xii}:
Combine \ref{t:25xi} and \cite[Proposition~17.5]{Livre1}.

\ref{t:25xiii}:
In this case, $L$ is an isometry and the conclusion follows from
\cite[Proposition~5.3(vii)]{Svva23}, \eqref{e:nn5j}, and
\eqref{e:4cpy}.
\end{proof}

\begin{remark}
The motivation for calling $\rmi(\LW,\fw)_{\omega\in\Omega}$ an
integral proximal mixture comes from
Theorem~\ref{t:25}\ref{t:25ix}.
\end{remark}

\begin{example}
\label{ex:30}
Consider the setting of Example~\ref{ex:1i+}. Then \eqref{e:120}
becomes
\begin{equation}
\label{e:64}
\rmi(\LS_k,\mathsf{f}_k)_{1\leq k\leq p}
=\brk3{\sum_{k=1}^p\alpha_k
(\mathsf{f}_k^*\infconv\qq_{\HS_k})
\circ\mathsf{L}_k}^*-\qq_{\XS},
\end{equation}
which is the \emph{proximal mixture} introduced in
\cite[Example~5.9]{Svva23}.
\end{example}

Our next illustration concerns a new object: the proximal 
expectation of a family of functions.

\begin{definition}
\label{d:201}
Let $(\Omega,\FF,\PP)$ be a complete probability space,
let $\HS$ be a separable real Hilbert space,
and let $(\fw)_{\omega\in\Omega}$ be a family of functions
in $\Gamma_0(\HS)$ such that the function
\begin{equation}
\Omega\times\HS\to\RX\colon(\omega,\mathsf{x})\mapsto
\fw(\mathsf{x})
\end{equation}
is $\FF\otimes\BE_{\HS}$-measurable.
Suppose that there exist
$r\in\mathscr{L}^2\brk{\Omega,\FF,\PP;\HS}$ and
$r^*\in\mathscr{L}^2\brk{\Omega,\FF,\PP;\HS}$ such that the
functions
$\omega\mapsto\fw(r(\omega))$ and
$\omega\mapsto\fw^*(r^*(\omega))$ lie in
$\mathscr{L}^1(\Omega,\FF,\PP;\RR)$.
Using the notation \eqref{e:vk}, the \emph{proximal
expectation} of the family $(\fw)_{\omega\in\Omega}$ is 
\begin{equation}
\label{e:fr2}
\PE(\fw)_{\omega\in\Omega}
=\brk2{\EE\brk!{\fw^*\infconv\qq_{\HS}}_{\omega\in\Omega}}^*
-\qq_{\HS}.
\end{equation}
\end{definition}

\begin{proposition}
\label{p:18}
Consider the setting of Example~\ref{ex:1iii} with the additional
assumption that $\mu$ is a probability measure.
Let $(\fw)_{\omega\in\Omega}$ be a family of functions
in $\Gamma_0(\HS)$ such that the function
\begin{equation}
\Omega\times\HS\to\RX\colon(\omega,\mathsf{x})\mapsto
\fw(\mathsf{x})
\end{equation}
is $\FF\otimes\BE_{\HS}$-measurable. Suppose that there exist
$r\in\mathscr{L}^2\brk{\Omega,\FF,\mu;\HS}$ and
$r^*\in\mathscr{L}^2\brk{\Omega,\FF,\mu;\HS}$ such that the
functions
$\omega\mapsto\fw(r(\omega))$ and
$\omega\mapsto\fw^*(r^*(\omega))$ lie in
$\mathscr{L}^1(\Omega,\FF,\mu;\RR)$.
Then
\begin{equation}
\label{e:205}
\PE(\fw)_{\omega\in\Omega}
=\rmi(\Id_{\HS},\fw)_{\omega\in\Omega}
=\rcm(\Id_{\HS},\fw)_{\omega\in\Omega}.
\end{equation}
\end{proposition}
\begin{proof}
Note that
\begin{equation}
\GG=\mathscr{L}^2\brk!{\Omega,\FF,\mu;\HS}.
\end{equation}
Using the completeness of $(\Omega,\FF,\mu)$, we derive from
\cite[Th\'eor\`eme~2.3]{Atto79}, \cite[Lemma~III.14]{Cast77}, and 
\cite[Proposition~12.28]{Livre1} that, for every $x\in\GG$, the
mapping $\Omega\to\HS\colon\omega\mapsto\prox_{\fw}x(\omega)$ lies
in $\mathfrak{G}$. Thus, for every $x^*\in\GG$, using
\cite[Theorem~14.3(ii)]{Livre1}, we deduce that the mapping
$\Omega\to\HS\colon\omega\mapsto\prox_{\fw^*}x^*(\omega)
=x^*(\omega)-\prox_{\fw}x^*(\omega)$ also lies in $\mathfrak{G}$.
Hence, the family $(\fw)_{\omega\in\Omega}$ satisfies
Assumption~\ref{a:3}. Thus, invoking Notation~\ref{n:62}, we
deduce from Theorem~\ref{t:25}\ref{t:25xiii},
\eqref{e:120}, and \eqref{e:fr2} that
\begin{align}
\rcm(\Id_{\HS},\fw)_{\omega\in\Omega}
&=\rmi(\Id_{\HS},\fw)_{\omega\in\Omega}
\nonumber\\
&=\brk3{\int_{\Omega}\brk!{(\fw^*\infconv\qq_{\HS})\circ\Id_{\HS}}
\,\mu(d\omega)}^*-\qq_{\HS}
\nonumber\\
&=\brk2{\EE\brk!{\fw^*\infconv\qq_{\HS}}_{\omega\in\Omega}}^*
-\qq_{\HS}
\nonumber\\
&=\PE(\fw)_{\omega\in\Omega},
\end{align}
as announced.
\end{proof}

Combining Theorem~\ref{t:25}, Proposition~\ref{p:101}, and
Proposition~\ref{p:18} yields at once the following properties of
the proximal expectation.

\begin{proposition}
\label{p:19}
Consider the setting of Definition~\ref{d:201}.
Then the following hold:
\begin{enumerate}
\item
$\PE(\fw)_{\omega\in\Omega}\in\Gamma_0(\HS)$.
\item
Let $\mathsf{x}\in\HS$. Then
\begin{multline}
\hspace{-10mm}\bigl(\PE(\fw)_{\omega\in\Omega}\bigr)(\mathsf{x})\\
\hspace{-12mm}=\min\Menge3{\int_{\Omega}\brk2{\fw\brk!{x(\omega)}+
\qq_{\HS}\brk!{x(\omega)}}\PP(d\omega)-\qq_{\HS}(\mathsf{x})}
{x\in L^2\brk!{\Omega,\FF,\PP;\HS}\,\,\text{and}\,\,
\EE x=\mathsf{x}}.
\end{multline}
\item
$\cdom\PE(\fw)_{\omega\in\Omega}=\overline{\menge{\EE x}{
x\in L^2\brk{\Omega,\FF,\PP;\HS}\,\,
\text{and}\,\,(\forallmu\omega\in\Omega)\,\,
x(\omega)\in\dom\fw}}$.
\item
$(\PE(\fw)_{\omega\in\Omega})^*
=\PE(\fw^*)_{\omega\in\Omega}
=\brk{\EE\brk{\fw\infconv\qq_{\HS}}_{\omega\in\Omega}}^*
-\qq_{\HS}$.
\item
$\partial\PE(\fw)_{\omega\in\Omega}=
\PE(\partial\fw)_{\omega\in\Omega}$.
\item
$\PE(\fw)_{\omega\in\Omega}\infconv\qq_{\HS}=
\EE(\fw\infconv\qq_{\HS})_{\omega\in\Omega}$.
\item
$\prox_{\PE(\fw)_{\omega\in\Omega}}=
\EE(\prox_{\fw})_{\omega\in\Omega}$.
\item
$\Argmin\PE(\fw)_{\omega\in\Omega}=
\Argmin\EE(\fw\infconv\qq_{\HS})_{\omega\in\Omega}$.
\end{enumerate}
\end{proposition}

\begin{remark}
\label{r:12}
In Definition~\ref{d:201}, consider the measure space 
$(\Omega,\FF,\mu)$ of Example~\ref{ex:1i+} with the additional 
assumption that $\sum_{k=1}^p\alpha_k=1$. Then the proximal 
expectation becomes
\begin{equation}
\label{e:fr3}
\PE(\mathsf{f}_k)_{1\leq k\leq p}
=\brk3{\sum_{k=1}^p\alpha_k\brk!{\mathsf{f}_k^*\infconv
\qq_{\HS}}}^*-\qq_{\HS}.
\end{equation}
\begin{enumerate}
\item
\label{r:12i}
The function of \eqref{e:fr3} is the \emph{proximal average} of 
the family $(\mathsf{f}_k)_{1\leq k\leq p}$. By specializing
Proposition~\ref{p:19} to this setting, we recover 
properties of the proximal average found in \cite{Baus08}.
\item
\label{r:12ii}
Let $(\mathsf{f}_k)_{1\leq k\leq p}$ be functions in
$\Gamma_0(\HS)$. 
In some data analysis applications (see, e.g., 
\cite{Cheu17,Huzz22,Yuyl13}),
\eqref{e:fr3} has been used instead of the standard average
$\sum_{k=1}^p\alpha_k\mathsf{f}_k$. The latter can be regarded as
the empirical $p$-sample approximation to the true expectation
$\EE(\fw)_{\omega\in\Omega}$ arising from a family 
$(\fw)_{\omega\in\Omega}$ in $\Gamma_0(\HS)$. Likewise, we can
regard the proximal average \eqref{e:fr3} as the empirical
approximation to the proximal expectation
$\PE(\fw)_{\omega\in\Omega}$.
\end{enumerate}
\end{remark}

\begin{remark}
The strategy described in Remark~\ref{r:12}\ref{r:12ii} can be
generalized as follows. Let $\mu$ be a probability measure. Then it
may be appropriate in certain variational problems to replace the
standard composite average $\int_{\Omega}(\fw\circ\LW)\mu(d\omega)$
by the integral proximal comixture
$\rcm(\LW,\fw)_{\omega\in\Omega}$ of
Definition~\ref{d:7}. The latter is easier to handle numerically as
its proximity operator is explicitly given by
Theorem~\ref{t:25}\ref{t:25x} and it follows from
Theorem~\ref{t:25}\ref{t:25xii} that its set of minimizers
coincides with that of the function
$\int_{\Omega}((\fw\infconv\qq_{\HW})\circ\LW)\mu(d\omega)$. 
\end{remark}

\section{Relaxation of systems of monotone inclusions}
\label{sec:5}

We place our focus on the following general system of composite
monotone inclusions. 

\begin{problem}
\label{prob:10}
Suppose that Assumptions~\ref{a:2} and \ref{a:4} are in force and
that $\mathsf{V}\neq\{\mathsf{0}\}$ is a closed vector subspace of
$\XS$. The task is to 
\begin{equation}
\label{e:p1}
\text{find}\,\,\mathsf{x}\in\mathsf{V}\,\,\text{such that}\,\,
(\forallmu\omega\in\Omega)\,\,\mathsf{0}\in\AW(\LW\mathsf{x}).
\end{equation}
\end{problem}

The instantiations of Problem~\ref{prob:10} found in 
\cite{Svva23,Ibap21,Siim22} (see also \cite{Byrn12,Cens05} for
further special cases) correspond to the setting of
Example~\ref{ex:1i+} with finitely many inclusions, that is,
\begin{equation}
\label{e:p19}
\text{find}\,\,\mathsf{x}\in\mathsf{V}\,\,\text{such that}\,\,
\brk!{\forall k\in\{1,\ldots,p\}}\;\;
\mathsf{0}\in\mathsf{A}_k(\mathsf{L}_k\mathsf{x}).
\end{equation}
On the other hand, the instantiation of 
\cite{Butn95} corresponds to the setting of Example~\ref{ex:1iii}
where $\mu$ is a probability measure, $\mathsf{V}=\HS$ and, for
every $\omega\in\Omega$, $\AW$ is the normal cone operator of a
nonempty closed convex subset $\mathsf{C}_\omega$ of $\HS$ and
$\LW=\Id_{\HS}$, that is,
\begin{equation}
\label{e:p29}
\text{find}\,\,\mathsf{x}\in\HS\,\,\text{such that}\,\,
(\forallmu\omega\in\Omega)\,\,\mathsf{x}\in\mathsf{C}_{\omega}.
\end{equation}
The last problem is known as the stochastic convex feasibility
problem. Our formulation targets a much broader inclusion model
than those.

Of interest to us are the scenarios in which Problem~\ref{prob:10}
has no solution and must be replaced by a relaxed one which
furnishes meaningful solutions. We consider the following
relaxation which corresponds, in the special case of
Example~\ref{ex:1i+}, to that proposed in \cite{Siim22}.

\begin{problem}
\label{prob:11}
Suppose that Assumptions~\ref{a:2} and \ref{a:4} are in force and
that $\mathsf{V}\neq\{\mathsf{0}\}$ is a closed vector subspace of
$\XS$, and let $\gamma\in\RPP$. The task is to 
\begin{equation}
\label{e:p2}
\text{find}\;\:\mathsf{x}\in\XS\;\:\text{such that}\;\:
\mathsf{0}\in\brk2{\proxc{\proj_{\mathsf{V}}}
{\rcm(\LW,\gamma\AW)_{\omega\in\Omega}}}\mathsf{x}.
\end{equation}
\end{problem}

Let us examine the interplay between Problems~\ref{prob:10} and
\ref{prob:11}.

\begin{proposition}
\label{p:24}
Consider the settings of Problems~\ref{prob:10} and \ref{prob:11},
let $\mathsf{S}_1$ and $\mathsf{S}_2$ be their respective sets
of solutions, and set $\mathsf{W}=\proxc{\proj_{\mathsf{V}}}
{\rcm(\LW,\gamma\AW)_{\omega\in\Omega}}$. Then
the following hold:
\begin{enumerate}
\item
\label{p:24i}
$\mathsf{W}$ is maximally monotone.
\item
\label{p:24ii}
$J_{\mathsf{W}}=\proj_{\mathsf{V}}\circ
(\Id_{\mathsf{X}}+\int_{\Omega}(\LW^*\circ(J_{\gamma\AW}-\Id_{\HW})
\circ\LW)\mu(d\omega))\circ\proj_{\mathsf{V}}$.
\item
\label{p:24iii}
$\mathsf{S}_1$ and $\mathsf{S}_2$ are closed convex sets.
\item
\label{p:24iv}
Problem~\ref{prob:11} is an exact relaxation of
Problem~\ref{prob:10} in the sense that $\mathsf{S}_1\neq\emp$ 
$\Rightarrow$ $\mathsf{S}_2=\mathsf{S}_1$.
\item
\label{p:24v}
$\mathsf{S}_2=\zer(N_{\mathsf{V}}+
\int_{\Omega}(\LW^*\circ(\moyo{\AW}{\gamma})\circ\LW)
\mu(d\omega))$.
\end{enumerate}
\end{proposition}
\begin{proof}
Set $A=\leftindex^{\mathfrak{G}}{\int}_{\Omega}^{\oplus}
\AW\mu(d\omega)$ and
\begin{equation}
\label{e:rsmp2}
L\colon\XS\to\HH\colon\mathsf{x}\mapsto
\mathfrak{e}_{\LS}\mathsf{x}.
\end{equation}
Then \cite[Proposition~3.5]{Cana24} asserts that 
\begin{equation}
\label{e:e11}
J_{\gamma A}=\leftindex^{\mathfrak{G}}{\Int}_{\Omega}^{\oplus}
J_{\gamma\AW}\mu(d\omega)
\quad\text{and}\quad
\moyo{A}{\gamma}=
\leftindex^{\mathfrak{G}}{\Int}_{\Omega}^{\oplus}
\moyo{\AW}{\gamma}\,\mu(d\omega).
\end{equation}
In addition, it follows from \cite[Proposition~3.12]{Cana24} that 
$L$ is a well-defined bounded linear operator with adjoint
\begin{equation}
\label{e:v1}
L^*\colon\HH\to\XS\colon
x^*\mapsto\int_{\Omega}\LW^*x^*(\omega)\mu(d\omega)
\end{equation}
and such that $\norm{L}\leq 1$. We also recall from 
\eqref{e:kvfp} that
\begin{equation}
\label{e:r0}
\rcm(\LW,\gamma\AW)_{\omega\in\Omega}=\proxcc{L}{(\gamma A)}.
\end{equation}
Further, \eqref{e:p1} is equivalent to 
\begin{equation}
\label{e:p6}
\text{find}\;\:\mathsf{x}\in{\mathsf{V}}\;\:\text{such that}\;\:
0\in A(L\mathsf{x})
\end{equation}
and \eqref{e:p2} is equivalent to 
\begin{equation}
\label{e:p7}
\text{find}\;\:\mathsf{x}\in\XS\;\:\text{such that}\;\:
\mathsf{0}\in\Bigl(\proxc{\proj_{\mathsf{V}}}
{\bigl(\proxcc{L}{(\gamma A)}\bigr)}\Bigr)\mathsf{x}.
\end{equation}

\ref{p:24i}:
We derive from \cite[Theorem~3.8(i)]{Cana24} that $A$ is maximally 
monotone and hence from \cite[Theorem~4.5(ii)]{Svva23} that 
$\proxcc{L}{(\gamma A)}$ is likewise. In turn, 
\cite[Theorem~4.5(i)]{Svva23} and \eqref{e:r0} assert that
$\mathsf{W}=\proxc{\proj_{\mathsf{V}}}{(\proxcc{L}{(\gamma A)})}$
is maximally monotone.

\ref{p:24ii}: By invoking successively
\cite[Theorem~6.3(ii)]{Svva23}, \eqref{e:e11}, and
\eqref{e:v1}, we obtain 
\begin{align}
J_{\mathsf{W}}
&=\proj_{\mathsf{V}}\circ\bigl(\Id_{\XS}
+L^*\circ(J_{\gamma A}-\Id_{\HH})\circ L\bigr)
\circ\proj_{\mathsf{V}}\nonumber\\
&=\proj_{\mathsf{V}}\circ\brk3{\Id_{\mathsf{X}}+
\int_{\Omega}\bigl(\LW^*\circ(J_{\gamma\AW}-\Id_{\HW})
\circ\LW\bigr)\mu(d\omega)}\circ\proj_{\mathsf{V}}.
\end{align}

\ref{p:24iii}:
Use \eqref{e:p6}, \eqref{e:p7},
and \cite[Theorem~6.3(iii)]{Svva23}.

\ref{p:24iv}:
Combine \eqref{e:p7} and \cite[Theorem~6.3(v)]{Svva23}.

\ref{p:24v}:
Using \eqref{e:p7}, \cite[Theorem~6.3(vi)]{Svva23},
\eqref{e:e11}, and \eqref{e:v1}, we obtain
\begin{equation}
\mathsf{S}_2
=\zer\bigl(N_{\mathsf{V}}+L^*\circ(\moyo{A}{\gamma})\circ L\bigr)
=\zer\brk3{N_{\mathsf{V}}+\int_\Omega
\bigl(\LW^*\circ(\moyo{\AW}{\gamma})\circ\LW\bigr)\mu(d\omega)},
\end{equation}
which concludes the proof.
\end{proof}

We now present an algorithm to solve Problem~\ref{prob:11}.

\begin{proposition}
\label{p:22}
Suppose that Problem~\ref{prob:11} has a solution, 
let $(\lambda_n)_{n\in\NN}$ be a sequence in $\left]0,2\right[$
such that $\sum_{n\in\NN}\lambda_n(2-\lambda_n)=\pinf$, and let
$\mathsf{x}_0\in\mathsf{V}$. Iterate
\begin{equation}
\label{e:9}
\begin{array}{l}
\text{for}\;n=0,1,\ldots\\
\left\lfloor
\begin{array}{l}
\text{for $\mu$-almost every}\,\,\omega\in\Omega\\
\left\lfloor
\begin{array}{l}
y_n(\omega)=\LW\mathsf{x}_n\\
q_n(\omega)=y_n(\omega)-J_{\gamma\AW}y_n(\omega)\\
\end{array}
\right.\\
\mathsf{z}_n=\int_{\Omega}\LW^*(q_n(\omega))\mu(d\omega)\\
\mathsf{x}_{n+1}=\mathsf{x}_n-
\lambda_n\proj_{\mathsf{V}}\mathsf{z}_n.
\end{array}
\right.
\end{array}
\end{equation}
Then $(\mathsf{x}_n)_{n\in\NN}$ converges weakly to a solution to 
Problem~\ref{prob:11}.
\end{proposition}
\begin{proof}
Set $\mathsf{W}=\proxc{\proj_{\mathsf{V}}}
{\rcm(\LW,\gamma\AW)_{\omega\in\Omega}}$ and recall from
Proposition~\ref{p:24}\ref{p:24ii} that 
\begin{equation}
J_{\mathsf{W}}=\proj_{\mathsf{V}}\circ
\brk3{\Id_{\mathsf{X}}+\int_{\Omega}
\bigl(\LW^*\circ(J_{\gamma\AW}-\Id_{\HW})
\circ\LW\bigr)\mu(d\omega)}\circ\proj_{\mathsf{V}}.
\end{equation}
We derive from \eqref{e:9}, \eqref{e:e11}, and \eqref{e:v1} that
$(\mathsf{x}_n)_{n\in\NN}$ is generated by the proximal point
algorithm
\begin{equation}
(\forall n\in\NN)\quad
\mathsf{x}_{n+1}=\mathsf{x}_n+\lambda_n
(J_{\mathsf{W}}\mathsf{x}_n-\mathsf{x}_n).
\end{equation}
It then follows from \cite[Lemma~2.2(vi)]{Joca09} that
$(\mathsf{x}_n)_{n\in\NN}$ converges weakly to a point in
$\zer\mathsf{W}$, i.e., a solution to \eqref{e:p2}. 
\end{proof}

\begin{example}
\label{ex:5}
Let us specialize Problem~\ref{prob:10} to the scenario in which
\begin{multline}
(\forall\omega\in\Omega)\quad
\AW=\bigl(\Id_{\HW}-\TW+\mathsf{r}_\omega\bigr)^{-1}-\Id_{\HW},\\
\text{where}\;\;
\begin{cases}
\TW\colon\HW\to\HW\;\text{is firmly nonexpansive}\\
\mathsf{r}_{\omega}\in\HW.
\end{cases}
\end{multline}
Then \eqref{e:p1} becomes 
\begin{equation}
\label{e:p4}
\text{find}\,\,\mathsf{x}\in\mathsf{V}\,\,\text{such that}\,\,
(\forallmu\omega\in\Omega)\,\,
\TW(\LW\mathsf{x})=\mathsf{r}_{\omega}.
\end{equation}
This model has been considered in \cite{Siim22} in the setting of 
Example~\ref{ex:1i+}. There, $\Omega$ is a finite set and each 
$\mathsf{r}_\omega$ models the observation of an unknown signal 
$\mathsf{x}\in\HS$ through a Wiener system, i.e., the concatenation
of a nonlinear operator $\TW$ and a linear transformation $\LW$.  
Our framework allows us to extend it to models with a continuum of
observations. In this context, and \eqref{e:p2} yields the relaxed
problem 
\begin{equation}
\label{e:45}
\text{find}\:\;\mathsf{x}\in\mathsf{V}\,\,\text{such that}\,\,
\int_\Omega\LW^*\bigl(\TW(\LW\mathsf{x})-\mathsf{r}_\omega\bigr)
\mu(d\omega)\in\mathsf{V}^\bot.
\end{equation}
Furthermore, \eqref{e:9} becomes 
\begin{equation}
\label{e:94}
\begin{array}{l}
\text{for}\;n=0,1,\ldots\\
\left\lfloor
\begin{array}{l}
\text{for $\mu$-almost every}\,\,\omega\in\Omega\\
\left\lfloor
\begin{array}{l}
y_n(\omega)=\LW\mathsf{x}_n\\
q_n(\omega)=\TW y_n(\omega)-\mathsf{r}_{\omega}\\
\end{array}
\right.\\
\mathsf{z}_n=\int_{\Omega}\LW^*(q_n(\omega))\mu(d\omega)\\
\mathsf{x}_{n+1}=\mathsf{x}_n-
\lambda_n\proj_{\mathsf{V}}\mathsf{z}_n.
\end{array}
\right.
\end{array}
\end{equation}
\end{example}

\end{document}